\documentclass[notitlepage,11pt,psamsfonts]{amsart}

\usepackage{amsmath,amssymb,amsthm}
\usepackage[latin1]{inputenc}
\usepackage{verbatim}
\newtheorem{thm}{Theorem}
\newtheorem{lem}{Lemma}
\newtheorem{cor}{Corollary}
\newtheorem{prop}{Proposition}
\newtheorem{rem}{Remark}
\newtheorem{defi}{Definition}
\newcommand{\eps}{\varepsilon}
\newcommand{\R}{\mathbb{R}}
\newcommand{\C}{\mathbb{C}}
\newcommand{\N}{\mathbb{N}}
\newcommand{\Z}{\mathbb{Z}}
\newcommand{\fu}{\mathfrak{u}}

\renewcommand{\Re}{{\rm Re}\,}
\renewcommand{\Im}{{\rm Im}\,}
\renewcommand{\div}{{\rm div}\,}
\newcommand{\Supp}{{\rm Supp}\,}

\newcommand{\Int}{\displaystyle \int}

\def\dq{\Delta_q}

\def\tilde{\widetilde}
\def\hat{\widehat}

\begin{document}
\title[On the linear wave regime of the Gross-Pitaevskii equation]
{On the linear wave regime of the Gross-Pitaevskii equation}

\author[F. Bethuel]{Fabrice B\'ethuel}
\address[F. Bethuel]
{ Laboratoire J.-L. Lions UMR 7598, Universit{\'e} Pierre et Marie Curie,
175 rue du Chevaleret, 75013 Paris, France}
\email{bethuel@ann.jussieu.fr}
\author[R. Danchin]{Rapha\"el Danchin}
\address[R. Danchin]
{Universit\'e Paris-Est, LAMA, UMR 8050,
 61 avenue du G\'en\'eral de Gaulle,
94010 Cr\'eteil Cedex, France.}
\email{danchin@univ-paris12.fr}
\author[D. Smets]{Didier Smets}
\address[D. Smets]
{Laboratoire J.-L. Lions UMR 7598, Universit{\'e} Pierre et Marie Curie,
175 rue du Chevaleret, 75013 Paris, France}
\email{smets@ann.jussieu.fr}

\date\today
\begin{abstract}
We study a long wave-length asymptotics for the Gross-Pitaevskii equation corresponding to perturbation of a constant state of modulus one.  We exhibit lower bounds on the first occurence of possible zeros (vortices) and compare the solutions with the corresponding solutions to the linear wave equation or variants. The results rely on the use of the Madelung transform, which yields the hydrodynamical form of the Gross-Pitaevskii equation,  as well as  of an augmented system.     

\end{abstract}

\maketitle

\section{Introduction}\label{sect:intro}

The dynamics of the Gross-Pitaevskii equation
$$
i\frac{\partial \Psi}{\partial t} + \Delta \Psi = \Psi (|\Psi|^2-1)  
\leqno{(GP)}
$$
on $\R^N\times \R,$ for $N\geq 1$, with non-trivial limit conditions at infinity, 
exhibits a remarkable variety of special solutions and regimes.  The purpose
of this paper is to investigate one of these regimes, namely   perturbations
of constant maps of modulus one, which are obvious stationary solutions, in a
long-wave asymptotics.  In particular, we restrict ourselves to solutions
$\Psi$ which do not vanish, so that we may write $$ \Psi = \rho
\exp(i\varphi). $$ In the variables $(\rho,\varphi)$, $(GP)$ is turned into
the system  $$ \left\{
\begin{array}{l}
\displaystyle
\partial_t \rho + 2 \nabla \varphi\cdot\nabla \rho + \rho \Delta \varphi = 0,
\\[5pt]
\displaystyle
\rho \partial_t \varphi + \rho |\nabla \varphi|^2 - \Delta \rho =
 \rho (1-\rho^2).
\end{array}
\right.
$$
Setting $u=2\nabla \varphi$ leads to the 
hydrodynamical form of $(GP)$
\begin{equation}\label{eq:hydro}
\left\{
\begin{array}{l}
\displaystyle
\partial_t \rho^2 + {\rm div}(\rho^2 u) = 0,
\\
\displaystyle
\partial_t u + u\cdot\nabla u + 2 \nabla \rho^2 = 2\nabla \bigl(
\frac{\Delta \rho}{\rho}\bigr).
\end{array}
\right.
\end{equation}
If one neglects the right-hand side of the second equation, which is often referred to as the quantum pressure, system \eqref{eq:hydro} is similar to the Euler equation for a compressible fluid, with pressure law $p(\rho) = 2\rho^2$. In particular, the speed of sound waves near the constant solution 
$\Psi = 1$, that is $\rho = 1$ and $u = 0$, is given by
$$c_s = \sqrt{2}.$$

In order to specify the nature of our perturbation as well as of our long-wave
asymptotics we introduce a small parameter $\eps>0$ and set
\begin{equation}\label{scale}
\left\{
\begin{array}{l}
\displaystyle
\rho^2(x,t) = 1 + \frac{\eps}{\sqrt{2}}a_\eps(\eps x,\eps t),
\\
\displaystyle
u(x,t) = \eps u_\eps(\eps x, \eps t),
\end{array}
\right.
\end{equation}
so that system \eqref{eq:hydro} translates into 
\begin{equation}\label{eq:dynaslow}
\left\{
\begin{array}{l}
\displaystyle
\partial_t a_\eps + \sqrt{2}\,{\rm div}\,u_\eps = -\eps{\rm div}(a_\eps
u_\eps), \\
\displaystyle
\partial_t u_\eps + \sqrt{2}\,\nabla a_\eps = \eps \left(-u_\eps\cdot\nabla
u_\eps + 2 \nabla \biggl(
\frac{\Delta\sqrt{\sqrt 2+\eps a_\eps}}
{\sqrt{\sqrt 2+\eps a_\eps}}\biggr)\right).
\end{array}
\right.
\end{equation}
The l.h.s. of this system corresponds to the linear wave operator with speed
$\sqrt{2},$ whereas the r.h.s. contains terms of higher order derivatives, which
correspond to the dispersive nature of the Schr\"odinger equation (with infinite
speed of propagation).

Our first main result provides a lower bound
for the first occurrence of a zero of $\Psi.$
\begin{thm}\label{thm:un}
	Let $s> 1+\frac{N}{2}.$ There exists $C\equiv C(s,N)$ such that for any initial
	datum $(a^0_\eps,u^0_\eps)$ verifying 
	$(a^0_\eps,u^0_\eps) \in
	H^{s+1}\times H^s$ and $C\eps \|(a^0_\eps,u^0_\eps)\|_{H^{s+1}\times H^s} \leq 1$ there exists 
	$$
	T_\eps \geq \frac{1}{C\eps \|(a^0_\eps,u^0_\eps)\|_{H^{s+1}\times H^s}}
	$$
	such that system \eqref{eq:dynaslow} as a unique solution $(a_\eps,u_\eps) \in
	\mathcal{C}^0([0,T_\eps];H^{s+1}\times H^s)$ satisfying
	$$
	\|(a_\eps(\cdot,t),u_\eps(\cdot,t))\|_{H^{s+1}\times H^s}\leq C
	\|(a^0_\eps,u^0_\eps)\|_{H^{s+1}\times H^s}  \quad \text{and}\quad
	\frac12\leq\rho\left(\cdot\:,\frac{t}{\eps}\right) \leq 2
	$$
	whenever $t\in [0,T_\eps].$ 
\end{thm}

\begin{rem}\label{rem:coeur}
	i) 
	From the ansatz \eqref{scale}, the time scale of system \eqref{eq:dynaslow} is accelerated by a factor $\eps$ with respect to the time scale of system $\mathrm (GP).$  In terms of the Gross-Pitaevskii equation, the lower bound $T_\eps$ given in Theorem \ref{thm:un} translates therefore into the bound
	$$
	\mathcal{T}_\eps = \eps^{-1} T_\eps \geq 
	\frac{1}{C\eps^2 \|(a^0_\eps,u^0_\eps)\|_{H^{s+1}\times H^s}}.
	$$
	ii)
	A typical initial datum that Theorem \ref{thm:un} allows to handle is
	$$
	\Psi^0(x) = \sqrt{1+\frac{\eps}{\sqrt{2}}a^0(\eps x)} \exp(i\varphi^0(\eps x)),
	$$
	where $u^0\equiv 2\nabla \varphi^0$ and $a^0$ do not depend on $\eps$ and belong to $H^{s+1}\times H^s.$ This corresponds to pertubations of the constant map $1$ of order $\eps$ for the modulus and of wave-length of order $\eps^{-1}.$ In this case, we obtain  the lower bound $T_\eps \geq \frac{c}{\eps},$ that is $ \mathcal{T}_\eps \geq \frac{c}{\eps^2}.$ 
\end{rem}

As a byproduct of Theorem \ref{thm:un}, treating the r.h.s 
of \eqref{eq:dynaslow} as a  perturbation,
 we deduce the following comparison estimate with loss of
three derivatives:

\begin{thm}\label{thm:deux}
	Let $s$, $a^0_\eps$ and $u^0_\eps$ be as in Theorem \ref{thm:un} and let
	$(\mathfrak a,\mathfrak u)$ denote the solution of the free wave equation
$$
\left\{
\begin{array}{l}
\displaystyle
\partial_t \mathfrak a + \sqrt{2}\, {\rm div}\,\mathfrak u = 0\\
\displaystyle
\partial_t \mathfrak u + \sqrt{2}\, \nabla \mathfrak a = 0,
\end{array}
\right.
$$ 
	with initial datum $(a^0_\eps,u^0_\eps).$ If $\eps\leq1$ then for $0\leq t \leq
	T_\eps$ we have
	$$
	\|(a_\eps,u_\eps)(t) - (\mathfrak a,\mathfrak u)(t)\|_{H^{s-2}} \leq C\left[ \eps t
	\|(a^0_\eps,u^0_\eps)\|_{H^{s+1}\times H^s}^2 + \eps^2 t
	\|(a^0_\eps,u^0_\eps)\|_{H^{s+1}\times H^s} \right]. 
	$$
\end{thm}

In Theorem \ref{thm:un}, the fact that $(a^0_\eps,u^0_\eps) \in	H^{s+1}\times
H^s$ with $s\geq 0$ implies in particular that the Ginzburg-Landau energy
$E(\Psi^0)$ of the
corresponding  function $\Psi^0$ is finite, where
$$
E(\Psi) = \int_{\R^N} e(\Psi) = \int_{\R^N}\bigl[\frac{1}{2}|\nabla \Psi|^2 +
\frac{1}{4}(1-|\Psi|^2)^2\bigr]
$$  
is the Hamiltonian for $(GP).$ 

Notice that according to \cite{PG}, the Cauchy problem for $(GP)$ is
globally well-posed in the energy space, in dimension $N=2,\: 3.$ 
On the other
hand, by means of a basic energy method, it may be easily seen that
 $(GP)$ is locally well-posed in $1+H^s$ in any dimension
provided $s>\frac{N}{2}.$ In addition, in both cases, the Ginzburg-Landau energy 
$E(\Psi)$ remains  conserved  during the evolution.

\smallbreak 
In dimension $N\geq2,$ in  order to handle longer time scales, one may take advantage
of the dispersive properties of
system \eqref{eq:dynaslow}. 
As a matter of fact, the linearization about $(0,0)$ of the system  \eqref{eq:dynaslow} does
not exactly  yield the wave operator, as appearing in Theorem \ref{thm:deux}, 
but rather the $\eps$-depending  operator
$$
L_\eps(a,u) = \left(\partial_t a + \sqrt{2}\,{\rm div} u, \partial_t u +
\sqrt{2}\,\nabla a - \sqrt{2}\, \eps^2 \nabla \Delta a\right),
$$
which possesses even better dispersive properties. 
Indeed, performing a Fourier transform with respect to the space variables, the above operator  rewrites for  
$\xi\in\R^N$  and $t \in \R$ as
$$
\hat{ L_\eps (  a, u )}(\xi, t)
= \begin{pmatrix}\partial_t\hat a(\xi,t)\\ \partial_t\hat u(\xi,t)  \end{pmatrix}
+i\begin{pmatrix} 0& \sqrt 2 \,\xi^T\\
\bigl(\sqrt 2+\sqrt2\,\eps^2|\xi|^2\bigr)\xi&0\end{pmatrix}
\begin{pmatrix} \hat a(\xi,t)\\\hat u(\xi,t)\end{pmatrix}.
$$
If we restrict our attention to potential solutions, that is solutions for which 
$u$ is a gradient,  then 
the eigenvalues associated to the above system are
$$
\lambda_\pm=\pm i\sqrt{2}|\xi|\sqrt{\eps^2|\xi|^2+1}.
$$
Therefore, we expect  $L_\eps$
to behave as the  linear wave operator
with velocity $\sqrt2$ 
for low frequencies $|\xi|\ll\eps^{-1}$ 
whereas for high frequencies $|\xi|\gg\eps^{-1},$
it should resemble  the linear Schr\"odinger equation with small  diffusion coefficient equal to 
$\sqrt2\,\eps.$
We thus expect to glean some additional smallness for the solution to the nonlinear equation 
\eqref{eq:dynaslow}
by resorting to the dispersive properties of
those two linear equations\footnote{Note however, that
 since no dispersion occurs for the wave equation in dimension $N=1,$
our method does not give any additional information on that case.}. 
This will enable us to improve the
lower bound for $T_\eps$ stated in Theorem \ref{thm:un} assuming 
 the dimension $N$
is larger than or equal to two. More precisely, we  prove the following statement. 
\begin{thm}\label{thm:trois}
	 Under the assumptions of Theorem \ref{thm:un} with
$s> 2 +\frac{N}{2}$ and $\eps\leq1,$  the time $T_\eps$ may be bounded from below by 
	$$
	\begin{array}{lll}
		\displaystyle
		\frac{c}{\eps^2 \|(a^0_\eps,u^0_\eps)\|_{H^{s+1}\times H^s}^2}
&  \text{if } N\geq 		4,\\[15pt]
		\displaystyle 
		\min\biggl(\frac{c}{\eps^{1+\alpha}
		\|(a^0_\eps,u^0_\eps)\|_{H^{s+1}\times H^s}^{1+\alpha}}, 
\frac{1}{\eps^3
		\|(a^0_\eps,u^0_\eps)\|_{H^{s+1}\times H^s}^2}\biggr) &
\text{if } N = 3\ \hbox{ and }\ 0<\alpha<1,\\[15pt] 		\displaystyle
	       \min\biggl(\frac{c}{\eps^{\frac43}
		\|(a^0_\eps,u^0_\eps)\|_{H^{s+1}\times H^s}^\frac{4}{3}},
		\frac{1}{\eps^{q+1}\|(a^0_\eps,u^0_\eps)\|^{q}
_{H^{s+1}\times H^s}}\biggr) & \text{if } N = 2\ \hbox{ and }\
		2>q>\frac{2}{s-2}.
	\end{array}
	$$

  The constant $c$ depends only on $s$
and also on $N$ if $N\geq4,$ $\alpha$ if $N=3$ and $q$ if $N=2.$
\end{thm}

\begin{rem}
	\label{rem:2coeurs}
	With an  initial datum as in Remark \ref{rem:coeur} ii), we obtain, as $\eps \to 0,$ $T_\eps \geq {c}{\eps^{-2}}$ if $N\geq 4,$ $T_\eps \geq {c}{\eps^{-(2^-)}}$ if $N\geq 3$, and $T_\eps \geq {c}{\eps^{-\frac{4}{3}}}$ if $N=2.$
\end{rem}

\begin{rem}
In dimension $1$ and $2$, the Gross-Pitaevskii equation is known to have travelling wave solutions $\psi(x,t)=W_\eps(x-c_\eps t)$ which are small amplitude and long wavelength perturbations of the constant $1.$ They are of the form
\begin{eqnarray*}
W_\eps(x) = 1 + \eps^2 w_\eps(\eps x)&\quad\text{in dimension 1},\\
\text{and}\\
W_\eps(x) = 1 + \eps^2 w_\eps(\eps x_1,\eps^2 x_2)&\quad\text{in dimension 2},
\end{eqnarray*}
where the speed $c_\eps$ is given by $c_\eps^2=2-\eps^2,$ and where $w_\eps$ remains bounded in strong norms as $\eps \to 0.$ For initial data of this form (but not necessarily the travelling waves), the corresponding $a_\eps^0$ and $u_\eps^0$ satisfy
$$
\|a_\eps^0,u_\eps^0\|_{H^{s+1}\times H^s} = \left\{ 
\begin{array}{ll}
	O(\eps)&\text{ if } N=1,\\
	O(\sqrt{\eps})&\text{ if }N=2.
\end{array}\right.
$$
If $N=1$, Theorem \ref{thm:un} shows that $\mathcal{T}_\eps \geq C\eps^{-3}$, and Theorem \ref{thm:trois} shows similarily that $\mathcal{T}_\eps \geq C\eps^{-3}$ when $N=2.$ In view of Theorem \ref{thm:deux}, the wave equation is a good approximation on time scales small with respect to $\eps^{-3}.$ For times of order $\eps^{-3},$ the wave equation is no longer a good approximation, as can be seen considering the travelling waves. Indeed since the speed of the travelling wave differs from the speed of sound $\sqrt{2}$ by an amout of order $\eps^2$, both solutions are shifted (in the variables for \eqref{eq:dynaslow}) by an amount of order 1 exactly after a time of order $\eps^{-2},$ which corresponds to a time of order $\eps^{-3}$ in the time variable of $(GP).$\\
For such timescales, one is lead to consider nonlinear approximations such as the KdV or the KP equations (see \cite{BGSS2,CR2}).
\end{rem}

\begin{rem}\label{rem:3coeurs}
	It may be worthwhile to compare these existence results with the corresponding ones for the irrotational compressible Euler equation with smooth compactly supported perturbations of size or order $\eps$ of a constant state. In that case, the corresponding $T_\eps$ is known to be $T_\eps = +\infty$ for $N\geq 4$, $T_\eps \geq \exp(\frac{c}{\eps})$ for $N=3$, $T_\eps \geq {c}{\eps^{-2}}$ for $N=2$ and $T_\eps \geq {c}{\eps}^{-1}$ for $N=1.$ (see e.g. \cite{Sid1,Sid2,Hor} following pioneering ideas by Klainerman \cite{Kla}).      
\end{rem}

On the larger time scale given by Theorem \ref{thm:trois}, equation
\eqref{eq:dynaslow} is better approximated by the linear equation
$L_\eps(a,u)=0$ than by the free wave equation.  More precisely, we have

\begin{thm}\label{thm:quatre}
Let $s>2+\frac{N}{2}$ and  $(a_\eps^0,u_\eps^0)$ be as in Theorem \ref{thm:trois}, let 
$(a_\eps,u_\eps)$ be the corresponding maximal solution of \eqref{eq:dynaslow}
and  $(\mathfrak{a}_{\eps},\mathfrak{u}_{\eps})$ be  the solution to the system
$$L_\eps(\mathfrak{a}_\eps,\mathfrak{u}_\eps)=0
  \quad\text{with initial datum }\  (a_\eps^0,u_\eps^0). 
$$
Let
$\alpha\in(0,\frac{1}{2})$ (satisfying also $\alpha>2-s/2$ if $N=2$).
 There exists a constant $C$ depending only on $s$,
$N$ and  possibly also on $\alpha$ if $N=2,3$ 
 such that for all
$t\in[0,T_\eps],$ the difference
$(\tilde a,\tilde u):=(a_\eps-\mathfrak{a}_{\eps},u_\eps-\mathfrak{u}_{\eps})$ satisfies
$$
\begin{array}{lll}
\|(\tilde a,\tilde u_\ell)(t)\|_{H^{s-1}} 
+\eps^{-1}\|\tilde u_h(t)\|_{H^{s-2}} \leq C\eps\sqrt{ t}\,
\|(a_\eps^0,u_\eps^0)\|^2_{H^{s+1}\times H^s}\ &\hbox{if}& \ N\geq4,\\[1ex]
\|(\tilde a,\tilde u_\ell)(t)\|_{H^{s-1}} 
+\eps^{-1}\|\tilde u_h(t)\|_{H^{s-2}} \leq
C\bigl(t^{1-\alpha}\eps+\eps^\frac{3}{2}\sqrt{t}\bigr)
 \|(a_\eps^0,u_\eps^0)\|^2_{H^{s+1}\times H^s}\ &\hbox{if}&\ N=3,\\[1ex]
\|(\tilde a,\tilde u_\ell)(t)\|_{H^{s-1}} 
+\eps^{-1}\|\tilde u_h(t)\|_{H^{s-2}} \leq C\bigl(\eps t^{\frac{3}{4}}+\eps^{2-\alpha} t^{1-\alpha}\bigr)
 \|(a_\eps^0,u_\eps^0)\|^2_{H^{s+1}\times H^s} \ &\hbox{if}&\ N=2.
\end{array}
$$   
\end{thm}

Here, $\tilde u_\ell$ and $\tilde u_h$ denote respectively the low and high
frequency parts of $\tilde u$, the threshold between the two being set once more at $\eps^{-1}$ (see the exact definition in \eqref{eq:def} below).

\medskip

In the existing mathematical literature, the Gross-Pitaevskii equation is sometimes considered in its semi-classical form
\begin{equation}\label{eq:gpsemi}
i\eps\frac{\partial \Psi_\eps}{\partial t} + \eps^2\Delta \Psi_\eps = \Psi_\eps (|\Psi_\eps|^2-1) . 
\end{equation}
One can easily recover the original equation $(GP)$ by mean of the hyperbolic scaling
$$
\Psi_\eps(x,t) = \Psi (\frac x \eps, \frac t \eps). 
$$
In this setting, we have
$$
\left\{
\begin{array}{l}
\displaystyle
a_\eps = \frac{\sqrt{2}}{\eps}(|\Psi_\eps|^2-1),
\\[8pt]
\displaystyle
u_\eps = 2 \nabla\left( \arg(\Psi_\eps)\right).
\end{array}
\right.
$$
In \cite{CoSo}, equation \eqref{eq:gpsemi} is considered on a bounded simply connected 
domain $\Omega \subset \R^2$ with Dirichlet boundary condition and initial datum of modulus one (so that $a_\eps$ vanishes at time zero), independent of $\eps$  and bounded in $H^1(\Omega).$ It is proved that $\Psi_\eps$ converges weakly in $L^\infty(\R_+,H^1(\Omega))$ and strongly in $\mathcal{C}^0([0,T],L^2(\Omega))$ to    $\Psi_*$ of modulus one whose phase satisfies the linear wave equation with speed $\sqrt{2}$. This is consistent with our result.  It is stronger in the sense that it allows for rough data, but it is also weaker in the sense that it only provides weak convergence.       

Another regime for \eqref{eq:gpsemi}, corresponding to oscillating phases, has been investigated by Grenier in \cite{Gr}, and more recently by Alazard and Carles \cite{AC}, Lin and Zhang \cite{LZ}, Zhang \cite{Z} and Chiron and Rousset \cite{CR}.  

Finally, the Gross-Pitaevskii equation has also been widely considered in a parabolic type scaling, namely
$$
\Upsilon_\eps(x,t) = \Psi(\frac x \eps,\frac t {\eps^2}),
$$  
so that $(GP)$ is turned into
\begin{equation}\label{eq:habit}
i\frac{\partial {\Upsilon_\eps}}{\partial t} + \Delta \Upsilon_\eps = \frac{1}{\eps^2}\Upsilon_\eps (|\Upsilon_\eps|^2-1) ,
\end{equation}
whose Hamiltonian reads
$$
E_\eps(\Upsilon) = \int_{\R^N}\bigl[\frac{1}{2}|\nabla \Upsilon|^2 +
\frac{1}{4\eps^2}(1-|\Upsilon|^2)^2\bigr].
$$  
Equation \eqref{eq:habit} is mainly considered in the regime where vortices are present \cite{JeCo,Lin,JeSp,BJS} and the energy is essentially reduced to the vortex energy so that no energy is left for wave oscillations as considered here.
As long as $\Psi$ does not vanish, equation $(GP)$ and the system
\eqref{eq:dynaslow} are obviously equivalent.  Therefore, Theorem
\ref{thm:un} yields a lower bound on the first occurrence of a zero of
$\Psi$ and hence of a vortex.
It would be of high interest to combine the two approaches in order to understand the interaction between these two different regimes.     

\medskip

System \eqref{eq:hydro} also enters in the class of  capillary fluid equations 
 studied in \cite{BDD}, with capillary coefficient $K(\rho)=\frac{1}{\rho}.$  Indeed,  we have
 \begin{equation}\label{eq:rho}
\frac{\Delta \rho}{\rho} = K(\rho^2)\Delta \rho^2 + K'(\rho^2) |\nabla \rho^2|^2
\ \text{ with }\ K(s)=\frac{1}{s}\cdot
\end{equation}
Notice  that, if we
 consider more general nonlinearities for $(GP)$, of the form $\Psi F(\vert \Psi \vert^2)$, the pressure is   turned into $p(\rho)=2F(\rho^2)$, whereas the capillarity coefficient remains unchanged.

\medskip

We now come to the main ingredients in the proofs of Theorem \ref{thm:un},
\ref{thm:deux}, \ref{thm:trois} and \ref{thm:quatre}.  For expository purposes,
it is convenient to use the parabolic 
scaling so as to  remove as much as possible the $\eps-$dependence.
More precisely, we introduce the new unknowns
$$
\left\{
\begin{array}{l}
	\displaystyle
	b_\eps(x,t) = a_\eps(x,\frac{t}{\eps})\\[8pt]
	\displaystyle
	v_\eps(x,t) = u_\eps(x,\frac{t}{\eps})
\end{array}
\right.
$$
so that the lower bound that we want to exhibit in Theorem \ref{thm:un}
becomes of order $1, $ for initial data as in Remark \ref{rem:coeur}.
\smallbreak
Notice that we have the relation 
$$
\Upsilon_\eps=\rho_\eps e^{i\varphi_\eps}\quad\hbox{with }\
\rho_\eps^2:=1+\frac{\eps}{\sqrt2}b_\eps\ \hbox{ and }\ 
v_\eps=2\nabla\varphi_\eps,
$$
and  that 
$(b_\eps,v_\eps)$ satisfies the system
\begin{equation}\label{eq:dyna}
\left\{
\begin{array}{l}
\displaystyle
\partial_t b_\eps + \frac{\sqrt{2}}{\eps}{\rm div}v_\eps = -{\rm
div}(b_\eps v_\eps),
\\[10pt]
\displaystyle
\partial_t v_\eps + \frac{\sqrt{2}}{\eps} \nabla b_\eps = -v_\eps \cdot\nabla
v_\eps + 2\nabla \bigl(
\frac{\Delta \rho_\eps}{\rho_\eps}\bigr).
\end{array}
\right.
\end{equation}

In view of the form of system
\eqref{eq:dyna}, our aim is to transpose the classical energy estimates for
symmetrizable hyperbolic systems. Indeed, in the linear case,
 the singular terms
involving $\frac{\sqrt{2}}{\eps}$ are transparent due to the skewsymmetry,
 and
do not contribute to the final balance.  However for the full system, in the
computation of the energy estimates, the higher order derivatives are difficult
to control, both by themselves and by their interaction with the previously
mentioned singular terms.  A similar difficulty in a related context was
overcome by S. Benzoni-Gavage, the second author  and 
S. Descombes in \cite{BDD}.  The crucial point there, inspired by earlier works
by F. Coquel \cite{Coquel}, is to consider an augmented system, adding the
equation for $\nabla (\log \rho_\eps^2).$ This choice is in fact quite natural since 
one may write
$$
\Upsilon_\eps=\exp\biggl(\frac{i}2\bigl(2\varphi_\eps-i\log\rho^2_\eps\bigr)\biggr).
$$
Therefore, we consider the new $\C^N$-valued function\footnote{Whenever  it  does
not lead to a confusion,  we omit the subscript  $\eps$.}
\begin{equation}\label{eq:sixbis}
z = v + i w \equiv \nabla (2\varphi - i \log \rho^2).
\end{equation}
We obtain the following system for the functions $z$ and $b$
\begin{equation}\label{eq:seven}
\left\{
\begin{array}{l}
\displaystyle
\partial_t b + \frac{\sqrt{2}}{\eps}{\rm div}({\rm  Re} z)
= -{\rm div}(b\, {\rm Re}z),\\
\displaystyle
\partial_t z + \frac{\sqrt{2}}{\eps} \nabla b =  i\Delta z - \nabla\Bigl(\frac{z\cdot
z}{2}\Bigr). 
\end{array}
\right.
\end{equation}
Here, for $z, z' \in \C^N$, we write $z\cdot z'= \sum_{k=1}^N z_kz_k'$ where the
products within the sum are complex multiplications.  
We first observe that
$$
\frac{\nabla\Upsilon_\eps}\Upsilon_\eps=\frac i2\,z\quad\text{and}\quad
|\Upsilon_\eps|^2-1=\frac{\eps^2}2b. $$ 
Therefore
$$
E(\Upsilon_\eps) = \frac1 8\biggl(\|b\|_{L^2(\R^N)}^2 +
\|z\|_{L^2(\R^N;(1+\eps b /\sqrt{2})dx)}^2\biggr).
$$
The main ingredient in the proof of Theorem \ref{thm:un} is
the following weighted a priori energy estimate involving high-order space derivatives:
\begin{prop}\label{prop:1}
Let $s$ be a nonnegative
 integer and let $\Upsilon_\eps$ be a
solution to $\eqref{eq:habit}$ such that  $(b,z) \in
\mathcal{C}^1([0,T],H^{s+1}(\R^N))$
and $(Db,Dz)\in \mathcal{C}^0([0,T];L^\infty)$
 for some $T>0.$  Assume that
\begin{equation}\label{eq:novortex}
m:=\inf_{x,t}|\Upsilon_\eps(x,t)| >0.
\end{equation}
 Then there exists a constant $C$ depending only
on $s,$ $m,$ $N,$ such for any time $t\in[0,T]$ we have for all
integer $s'\in\{0,\cdots,s\},$
$$
\frac{d}{dt} \Gamma^{s'}(b,z) \leq C (1+\eps\|b\|_{L^\infty})\|(Db,Dz)\|_{L^\infty} \left(
\Gamma^{s'}(b,z)
+ E_\eps(\Upsilon_\eps)\right),$$
where 
$$
\Gamma^s(b,z) = \|D^s b\|^2_{L^2(\R^N)} + \|D^s z\|_{L^2(\R^N;(1+\eps
b/\sqrt{2})dx)}^2.
$$
\end{prop}
\begin{rem}
A generalization of the above proposition
to noninteger  
Sobolev exponents and Besov spaces is given in Section \ref{ss:besov}.
 Notice that  for the case $s'=0,$ we have, in view of the  conservation of energy, the identity
$$
\frac{d}{dt}\Gamma^0(b,z)=0.
$$
\end{rem}

The main idea of the proof of Proposition \ref{prop:1} is that, up
to lower order terms which may be bounded with no loss of derivatives
provided $Db$ and $Dz$ are in $L^\infty,$ the structure for the system satisfied by
$(D^kb,D^kz)$ is the same as that of system \eqref{eq:seven}.
 For the proof of Theorem \ref{thm:un}, we  perform a time integration in the estimate of 
 Proposition \ref{prop:1}  which yields
$$
\|(b,z)(t)\|_{H^s}\leq
C\|(b^0,z^0)\|_{H^s}\exp\biggl(\int_0^t\|(Db,Dz)\|_{L^\infty}\,d\tau\biggr)
$$
whenever $1+\eps b/\sqrt2$ remains bounded and bounded away from zero.
In other words, the $H^s$ norms of 
$(b,z)(t)$ may be bounded in terms of the $H^s$ norms of the initial data
provided  we have a control over $(Db,Dz)$ in $L^1([0,t];L^\infty).$
If $s>N/2+1,$ it follows from the Sobolev embedding that  
$\|(Db,Dz)\Vert_{L^\infty}$ may be bounded by $\|(b,z)\|_{H^s}$
so that the above inequality leads to an explicit 
differential inequality for $\|(b,z)(t)\|_{H^s}$ and it is 
then straightforward to close the estimate for times
of order $\|(b_0,z_0)\|_{H^s}^{-1}.$

The proof of Theorem \ref{thm:deux} is based on elementary energy estimates for the system satisfied by    $(a_\eps,u_\eps)-(\mathfrak a,\mathfrak u),$  the source term of which being controlled thanks  to Theorem \ref{thm:un}. 
\smallbreak
 As mentioned above, the proofs of  Theorem \ref{thm:trois} 
and Theorem \ref{thm:quatre}  rely on dispersive properties 
of the equation. More precisely, we provide in Proposition \ref{p:dispersive} some  
Strichartz type estimates (in the spirit of  the pioneering work 
by R. Strichartz in \cite{Str} and of the paper
by J. Ginibre and G. Velo \cite{GV}) 
tailored for the operator $L_\eps$.  
Let us emphasize that related estimates have been 
used by the second author in \cite{Dan} 
for the study of slightly compressible fluids and 
by S. Gustafson, K. Nakanishi and T.P. Tsai
in  \cite{nakanishidanslacolle} for the Gross-Pitaevskii equation. 
 These estimates allow to 
  improve the  control on  the term $\|(Db,Dz)\|_{L^\infty}$
appearing in the key inequality of Proposition \ref{prop:1}. Indeed, it
 turns out that in dimension $N\geq2,$ 
one gets an additional 
bound for $\eps^{-\frac1p}\|(Db,Dz)\|_{L^p([0,t];L^\infty)}$
for some $p\in[2,\infty[$ depending on the dimension. 
 

\section{Short time existence and well-posedness for $(GP)$}

This section is devoted 
to the proof of  local well-posedness 
for $(GP)$  with suitably smooth initial  data which 
bounded away from zero.  Since such data do not fit in the standard Sobolev space framework, 
we introduce, as in \cite{BS},  the class of maps
$$
\mathcal{V} = \left\{ U \in L^\infty(\R^N,\C), \nabla^k U \in L^2(\R^N),
\forall k\geq 2, \nabla |U| \in L^2(\R^N), (1-|U|^2)\in L^2(\R^N)\right\}. 
$$ 
A first short time existence result is given by
\begin{prop}\label{prop:BS}
i) Let $U \in \mathcal{V}$ and $s>{\rm Max}(1,N/2).$ The Cauchy problem for
$(GP)$ is locally well-posed in $U+H^s(\R^N).$ More precisely, given $R>0$
there exists a time $T(R)>0$ such that if $\|\Phi^0\|_{H^s}\leq R$ then there
exists a unique solution $t\mapsto \Psi(t)$ in
$\mathcal{C}^0([-T(R),T(R)];U+H^s(\R^N))$ 
satisfying the initial time condition
$$
\Psi(0) = U + \Phi^0.
$$   
ii) The flow map $\Phi^0 \mapsto \Phi:= \Psi-U$ is continuous 
from the ball $B(R)$ of
$H^s(\R^N)$ into $\mathcal{C}^0([-T(R),T(R)],U+H^s(\R^N)).$\\
iii) If $\Psi(0)\in U+H^{s+2}(\R^N),$ then $t\mapsto\Psi(t)$ belongs to
$\mathcal{C}^1([-T(R),T(R)];U+H^s(\R^N)).$\\
iv) If $E(\Psi(0))<+\infty,$ then
$$
\frac{d}{dt} E(\Psi(t)) = 0, \;\forall t \in (-T(R),T(R)).
$$
v) If $E(\Psi(0)) <+\infty,$ then
$$
\|\Psi(t)-\Psi(0)\|_{L^2(\R^N)} 
\leq C\exp(C|t|),\; \forall t \in (-T(R),T(R)),  
$$
where the constant $C$ depends only on $E(\Psi(0)).$
\end{prop}

 The proof of Proposition
\ref{prop:BS} statements {\it i)} to {\it iii)}
 is similar to that of \cite{BS}
Proposition 3, and follows directly from classical semi-group theory with
locally lipschitz nonlinearities (see e.g. \cite{CaHa} Section 4.3).  For the
proof of {\it iv)}
 we invoke the conservation of energy for sufficiently regular
solutions (say in $U+H^{s+2}(\R^N)$) and then pass to the limit using
well-posedness in $U+H^s(\R^N)$.  This only requires $s>1.$ 
For the proof of {\it v)},
we refer to \cite{BS} Lemma 3.

\begin{rem}\label{rem:maximaltime}
In view of Proposition \ref{prop:BS}, if $s>1+N/2$ then for $\Psi^0$ in
$\mathcal{V} + H^{s+1}(\R^N)$ there exists a maximal time of existence
$T^s(\Psi^0)$ and a unique solution $u \in
\mathcal{C}([0,T^s(\Psi^0)),\Psi^0+H^{s+1}(\R^N))$ such that
$\Psi(0)=\Psi^0.$ Moreover, either
$$
T^s(\Psi^0)=+\infty \quad\text{or}\quad \limsup_{t\to T^s(\Psi^0)}
\|\Psi(t)-\Psi(0)\|_{H^{s+1}(\R^N)} = +\infty,
$$   
and the map $\Psi^0 \mapsto T^s(\Psi^0)$ is upper semi continuous for the
$H^{s+1}$ distance. 
\end{rem}

\section{Proof of Proposition \ref{prop:1}, and related results}

Setting $X\equiv (b,\Re z,\Im z)\in\R^{2N+1},$ system \eqref{eq:seven} may be recast in a more
abstract form as
\begin{equation}\label{eq:poireau}
\partial_t X  = \sum_{j=1}^N A_\eps^j \partial_j X + N_\eps(X)
\end{equation}
where the $(2N+1)-$matrices $A_\eps^j$ are symmetric, and represent the linear
one order terms of the r.h.s. of the system, whereas $N_\eps$ stands for the
nonlinear and second order terms.  The matrices $A_\eps^j$ are constant, and
contain terms which diverge as $\eps^{-1}.$ If the term $N_\eps$ were not
present in \eqref{eq:poireau}, then one would have a linear symmetric hyperbolic
system, and therefore conservation of all the $H^k$ norms of $X.$  Indeed, if
\begin{equation}\label{eq:poivron}
\partial_t Y  = \sum_{j=1}^N A_\eps^j \partial_j Y 
\end{equation}
then,
$$
\frac{1}{2}\frac d{dt} \|D^kY\|_{L^2}^2 = \int_{\R^N}
\langle D^k\partial_tY,D^kY\rangle =
\sum_{j=1}^N\int_{\R^N} \langle A_\eps^j \partial_j D^kY,D^kY\rangle,
$$
and,
 using  the symmetry of the matrices, 
\begin{equation}\label{eq:poivron1}
\sum_{j=1}^N\int_{\R^N} \langle A_\eps^j \partial_j D^kY,D^kY\rangle
=
\sum_{j=1}^N \int_{\R^N}\langle \partial_j D^kY,A_\eps^j D^kY\rangle
=-\sum_{j=1}^N \int_{\R^N}\langle  D^kY,A_\eps^j\partial_j D^kY\rangle.
\end{equation}
Therefore  $\|D^kY\|_{L^2}^2$ is time independent.
\smallbreak
Owing to the additional term $N_\eps(X),$
proving Sobolev estimates (or even energy estimates)
for \eqref{eq:poireau} is more involved. 
The reason why is that  the function $N_\eps$ 
contains terms of rather different nature
from the ``algebraic'' point of view:
\begin{itemize}
\item   semi-linear first order
terms, namely $-{\rm Re}z\nabla b - b{\rm div}({\rm Re}z),$ and
$-\nabla(\frac{z\cdot z}{2}),$ 
\item  the linear second order term $-i\Delta z.$
\end{itemize}
It is  not clear however that adding 
this latter terms to \eqref{eq:poivron} would not change 
the computation in
\eqref{eq:poivron1}. 
To deal with the  semi-linear first order terms,
we will have to introduce the quantity
 $\Gamma^s(b,z)$ which is different from
$\|D^sX\|_{L^2}$ since the $z$ part is weighted by the weight $1+
\frac{\eps}{\sqrt{2}}b.$ This weight
 plays somehow the role of a symmetrizer.  
To control  its  influence 
(in particular on the second order term), we
invoke the relation between the weight and $z$, namely
\begin{equation}\label{eq:pota}
-\nabla\Bigl(1+\frac{\eps}{\sqrt{2}}b\Bigr) = \Bigl(1+\frac{\eps}{\sqrt{2}}b\Bigr){\rm Im}z
\end{equation}
which, in some sense,  represents a gain of one derivative. 
 When $s$ is an integer, the
computation is a little more explicit.  Therefore we present that case
first.

\subsection{Proof when $s$ in an integer}\label{ss:integer}

In this paragraph, we assume that $s=k$ for some $k\in\N.$ 
Throughout, it is understood that 
for $z_1\in\C^N$ and $z_2\in\C^N$ the notation
$\langle z_1,z_2\rangle$ stands for the inner product 
in $\R^{2N}$ between the vectors $(\Re z_1,\Im z_1)$ and $(\Re z_2,\Im z_2).$
We first compute the time derivative of $\Gamma^k(b,z),$  namely we have
\begin{equation}\label{eq:boutin}\begin{split}
\frac{d}{dt} &\int_{\R^N} (1+\frac{\eps}{\sqrt{2}}b)\langle D^kz,D^kz\rangle +
\langle D^kb,D^kb\rangle\\
&= 2 \int_{\R^N} (1+\frac{\eps}{\sqrt{2}}b)\langle D^kz,D^k\partial_t z\rangle +
\langle D^kb,D^k\partial_t b\rangle  +  \int_{\R^N}
\frac{\eps}{\sqrt{2}}\partial_t b \langle D^kz,D^k z\rangle\\
&= I_1+I_2+I_3.
\end{split}\end{equation}
\noindent
{\bf Step 1:} Expansion of $I_1$ and $I_2$. 

In $I_1+I_2$, we replace $\partial_t z$ and
$\partial_t b$ by their values according to \eqref{eq:seven},  and expand the corresponding
expressions.  This yields
$$
I_1=2(I_{1,1}+ I_{1,2}+I_{1,3}+I_{1,4}+I_{1,5}) \qquad\text{and}\qquad I_2=2(I_{2,1}+ I_{2,2})
$$
where 
\begin{align*}
&I_{1,1} = \int_{\R^N} \langle D^k z, D^k(-\frac{\sqrt{2}}{\eps}\nabla b)\rangle,
&&I_{2,1} = \int_{\R^N} \langle D^k b, D^k(-\frac{\sqrt{2}}{\eps}{\rm div}({\rm
Re}z))\rangle,\\
&I_{1,2} = \int_{\R^N} b \langle D^k z, D^k(-\nabla b)\rangle,
&&I_{2,2} = \int_{\R^N} \langle D^k b, D^k(-{\rm div}(b{\rm
Re}z))\rangle.\\
&I_{1,3} = \int_{\R^N} \langle D^k z, D^k(i\Delta z)\rangle,\\
&I_{1,4} = \int_{\R^N} \frac{\eps}{\sqrt{2}} b \langle D^k z, D^k(i\Delta z)\rangle,\\
&I_{1,5} = \int_{\R^N} (1+\frac{\eps}{\sqrt{2}} b) \Big\langle D^k z,
D^k\big(-\nabla(\frac{z\cdot z}{2})\big)\Big\rangle,
\end{align*}

\noindent
{\bf Step 2:} Both $I_{1,3}$ and $I_{1,1}+I_{2,1}$ vanish. 

This is a consequence of the properties of the linear part of the equation as
explained before.  It follows by  integration by parts, and, for
$I_{1,1}+I_{2,1}$, from the fact that $b$ is real valued.

\noindent
{\bf Step 3:} Estimates for $I_{1,2}+I_{2,2}.$ 

Integrating by parts in $I_{2,2}$ then using Leibniz formula, we obtain
\begin{align*}
I_{1,2} + I_{2,2} &= \int_{\R^N} \langle D^k(\nabla b), D^k(b{\rm Re}z) -
bD^kz\rangle\\
&= \int_{\R^N} \langle D^k(\nabla b),D^kb\, {\rm Re}z\rangle  + \sum_{j=1}^{k-1}
\int_{\R^N} \langle D^k(\nabla b),D^jbD^{k-j} {\rm Re}z\rangle\\
&=  \int_{\R^N} \langle \nabla \frac{|D^kb|^2}{2}, {\rm Re}z\rangle -
\sum_{j=1}^{k-1}\int_{\R^N} \langle D^k b,{\rm div}(D^jbD^{k-j} {\rm Re}z)\rangle\\
&= - \int_{\R^N} \frac{|D^kb|^2}{2}\, {\rm div}({\rm Re}z) -
\sum_{j=1}^{k-1}\int_{\R^N} \langle D^k b,{\rm div}(D^jbD^{k-j} {\rm
Re}z)\rangle.
\end{align*}
For the first term, we write
$$
\left|\int_{\R^N} |D^kb|^2\, {\rm div}({\rm Re}z)
\right| \leq \|Dz\|_{L^\infty} \|b\|_{H^k}^2.
$$
In order to bound  the second term , one may rely on  
 Lemma \ref{lem:A.1} in the Appendix which yields, for
$j=1,\cdots,k-1,$ $$
\int_{\R^N}| \langle D^k b,{\rm div}(D^jbD^{k-j}| \leq C
\|(Db,Dz)\|_{L^\infty}\left( \|b\|_{H^k}^2 + 
\|z\|_{H^k}^2\right).
$$
Combining the two last inequalities we obtain
\begin{equation}\label{eq:boutin2}
|I_{1,2}+I_{2,2}| \leq C \|(Db,Dz)\|_{L^\infty}\left( \|b\|_{H^k}^2 +
\|z\|_{H^k}^2\right).
\end{equation}

\noindent
{\bf Step 4:} Estimates for $2I_{1,4}+2I_{1,5}+I_3.$ 

The sum of
 these three terms presents a remarkable compensation.  Indeed, integrating by
parts in $I_{1,4}$ we obtain
$$
I_{1,4} =- \int_{\R^N} \frac{\eps}{\sqrt{2}} \nabla b \langle D^k z, D^k(i\nabla
z)\rangle, 
$$ 
where we used the pointwise identity $\langle D^k(\nabla z),D^k (i\nabla
z)\rangle = 0.$ 

Using identity \eqref{eq:pota}, we are led to
$$
I_{1,4} = \int_{\R^N} (1+ \frac{\eps}{\sqrt{2}}b) \langle D^k z, D^k(i\nabla
z)\rangle {\rm Im}z. 
$$
Next, we turn to $I_{1,5}.$  First, expanding $\nabla(\frac{z\cdot z}{2}),$ we
get
\begin{align*}
I_{1,5} &= -\int_{\R^N} (1+\frac{\eps}{\sqrt{2}} b) \langle D^k z,
D^k(z\cdot \nabla z)\rangle\\
&= -\int_{\R^N} (1+\frac{\eps}{\sqrt{2}} b) \langle D^k z,
D^k(\nabla z)\cdot z\rangle - \sum_{j=0}^{k-1} \int_{\R^N} (1+\frac{\eps}{\sqrt{2}} b) \langle D^k z,
D^j(\nabla z)\cdot D^{k-j}z\rangle\\
&= I_{1,5}' + I_{1,5}''.
\end{align*}
Relying  once more on  Lemma \ref{lem:A.1} of the Appendix, we obtain for
$j=0,\cdots,k-1,$
\begin{equation}\label{eq:boutin3}
I_{1,5}'' \leq C \bigl(1+ \eps \|b\|_{L^\infty}\bigr) \|Dz\|_{L^\infty}
\|z\|_{H^k}^2.
\end{equation}
To estimate the first term $I_{1,5}'$, we use the algebraic identity
$$
\langle z_1,\zeta\,z_2\rangle = \langle z_1,z_2\rangle {\rm Re}\,\zeta + \langle
z_1, i z_2\rangle {\rm Im}\,\zeta\qquad \forall z_1,z_2 \in \mathbb{C}^N,\forall
\zeta \in \mathbb{C}. 
$$
This yields for all $j\in\{1,\cdots,N\},$
\begin{align*}
(1+\frac{\eps}{\sqrt{2}}b)
\langle D^k z,D^k(\partial_jz)\cdot z^j\rangle & = 
(1+\frac{\eps}{\sqrt{2}}b)\left[ \langle D^k z,D^k(\partial_jz)\rangle
 {\rm Re}\,z^j + 
\langle D^k z,D^k(i\partial_jz)\rangle {\rm Im}\,z^j\right]\\
&= (1+\frac{\eps}{\sqrt{2}}b)\left[ {\rm Re}\,z^j \partial_j (\frac{|D^k z|^2}{2}) +
\langle D^k z,D^k(i\partial_jz)\rangle {\rm Im}\,z^j\right]
\end{align*}
so that, integrating by parts in the first integral, 
\begin{align*}
2\bigl(I_{1,5}'+I_{1,4}\bigr)+I_3 &= - \int_{\R^N} (1+\frac{\eps}{\sqrt{2}}b)\Re z\cdot \nabla
|D^k z|^2+  \int_{\R^N}
\frac{\eps}{\sqrt{2}}\partial_t b \langle D^kz,D^k z\rangle \\ 
&= \int_{\R^N} {\rm div}\biggl(\Bigl(1+\frac{\eps}{\sqrt{2}}b\Bigr) \Re  z\biggr)\,|D^kz|^2
+  \int_{\R^N}\frac{\eps}{\sqrt{2}}\partial_t b\,|D^kz|^2.
\end{align*}
Since system \eqref{eq:seven} is satisfied, one can now conclude that 
\begin{equation}\label{eq:boutin4}
2I_{1,5}'+2I_{1,4}+I_3=0.
\end{equation}
\noindent
{\bf Step 5:} Proof of Proposition \ref{prop:1} completed when $s$ is an
integer.

Under condition \eqref{eq:novortex}, there exists a constant $C$ depending only on 
$k,$ $m$ and such that 
\begin{equation}\label{eq:victor0}
\|(b,z)\|_{H^k}^2\leq C\bigl(E_\eps(\Upsilon_\eps)+\Gamma^k(b,z)\bigr).
\end{equation}
Hence, combining 
\eqref{eq:boutin}, \eqref{eq:boutin2}, \eqref{eq:boutin3} and \eqref{eq:boutin4}
completes the proof.\qed


\subsection{Generalization of Proposition \ref{prop:1}}\label{ss:besov}

In this section, we extend  Proposition \ref{prop:1}
to the case of Sobolev spaces with noninteger exponents.
The proof that we propose is based on a Littlewood-Paley decomposition
and actually covers the case of Besov spaces $B^s_{2,r}$ as well. 

We first recall the notion of Littlewood-Paley decomposition. 
Let $(\chi,\varphi)$ being smooth compactly supported functions
such that 
\begin{enumerate}
\item $\chi$ is supported in $B(0,4/3)$, 
\item $\varphi$ is supported in the annulus $C(0,3/4,8/3)$,
\item $\forall\xi\in\R^N,\;
\chi(\xi)+\sum_{q\in\N}\varphi(2^{-q}\xi)=1$.
\end{enumerate} 
We denote\footnote{According to a classical convention,
$\psi(D)$ will  stand for the Fourier multiplier  of symbol
$\psi(\xi)$.}  $S_q:=\chi(2^{-q}{\rm D}),$
$\dq:=\varphi(2^{-q}{\rm D})$ for
$q\in\N,$ and $\Delta_{-1}:=S_0=\chi({\rm D}).$
We have 
$S_q=\sum_{p=-1}^{q-1}\Delta_p$ and 
$u=\sum_{q\geq-1}\dq u$
whenever $u$ is in ${\mathcal S}'(\R^N)$.  
Moreover, we have 
\begin{equation}\label{quasiorthogonality}
|p-q|>1\Longrightarrow
\dq\Delta_pu=0\quad\text{and}\quad|p-q|>4\Longrightarrow
\dq(S_{p-1}u\Delta_pv)=0.
\end{equation}
The Littlewood-Paley decomposition is defined by the identity
$$
u=\sum_{q\geq-1}\Delta_qu
$$
and makes sense for arbitrary  tempered distributions.
Furthermore, it is not difficult to check that $H^s(\R^N)$ coincides with 
the space of tempered distributions $u$ such that 
$$
\biggl(\sum_{q\geq-1}2^{2qs}\|\dq u\|_{L^2}^2\biggr)^{\frac12}<\infty
$$
and the left-hand side  of this inequality defines a norm on $H^s(\R^N)$ which 
is equivalent to the usual one. 
More generally,  one can  define the Besov space $B^s_{2,r}(\R^N)$
as the set of tempered distributions $u$ 
such that
$$
\|u\|_{B^s_{2,r}}:=\bigl\|2^{qs}\|\dq u\|_{L^2}\|_{\ell^r}<\infty.
$$
For $r=2$, we recover the usual Sobolev
space since $H^s(\R^N)= B^s_{2,2}(\R^N)$ with equivalent norms.
\medbreak
The remainder of this  section is devoted to the proof 
of the following proposition.
\begin{prop}\label{prop:1besov}
Let $s>0$ and $r\in[1,\infty].$
Assume that  $\Upsilon_\eps$ is a
solution to $\eqref{eq:habit}$ such that $(b,z) \in
\mathcal{C}^1([0,T];B^{s+1}_{2,r})\cap\mathcal{C}^0([0,T];W^{1,\infty})$ 
for some $T>0,$ and that
$m:=\inf_{x,t}|\Upsilon_\eps(x,t)| >0.$ 
 There exists a constant $K$ depending only
on $m,$ $s$ and $N$ such for any time $t\in[0,T]$ we have
\begin{equation}
\label{eq:besov1}
\frac d{dt}\int \Bigl(1+\frac{\eps}{\sqrt2}b\Bigr)2^{2qs}|\dq z|^2
+2^{2qs}|\dq b|^2\leq 
Kc_q(1+\eps\|b\|_{L^\infty})
\|(Db,Dz)\|_{L^\infty}
\|(Db,Dz)\|_{B^{s-1}_{2,r}}
\end{equation}
where the sequence $(c_q)_{q\geq-1}$
satisfies $\|(c_q)\|_{\ell^r}=1.$
\end{prop}
\begin{rem}\label{r:besov}
Remark that if we assume that 
$$
|\Upsilon_\eps(x,t)|^{\pm1}\leq M\quad\hbox{for all }\ (x,t)\in\R^N\times[0,T],
$$
then a $\ell^r$ 
 summation and a time integration in \eqref{eq:besov1}
  implies that we have for some constant
  $K$ depending only on $M,$ $s$ and $N,$
 $$
\|(b,z)(t)\|_{B^s_{2,r}}
 \leq K \biggl(\|(b,z)(0)\|_{B^s_{2,r}}
 +\int_0^t \|(Db,Dz)(\tau)\|_{L^\infty}\|(b,z)(\tau)\|_{B^s_{2,r}}\,d\tau\biggr).
 $$
 In particular, taking $r=2$ yields
 $$\|(b,z)(t)\|_{H^s}
 \leq K\biggl( \|(b,z)(0)\|_{H^s}
 +\int_0^t \|(Db,Dz)(\tau)\|_{L^\infty}\|(b,z)(\tau)\|_{H^s}\,d\tau\biggr).
 $$ 
\end{rem}
\begin{proof}
The proof works follows almost the same lines as the case  in the Sobolev case 
with integer exponents:  the main point is  to replace
the differential operator $D^k$ by 
the  Littlewood-Paley operator $\dq.$
Throughout the  computation, several  commutators
will appear, which may be dealt with thanks
to Lemma \ref{l:alamormoilnoeud}. The starting point is the following computation 
\begin{equation*}\begin{split}
\frac{d}{dt} &\int_{\R^N} \biggl\{(1+\frac{\eps}{\sqrt{2}}b)\langle\dq z,\dq
z\rangle + |\dq b|^2\biggr\}\\
&= 2 \int_{\R^N} \biggl\{(1+\frac{\eps}{\sqrt{2}}b)\langle \dq z,\dq\partial_t
z\rangle + \dq b\,\dq\partial_t b  \biggr\}+  \int_{\R^N}
\frac{\eps}{\sqrt{2}}\partial_t b \langle\dq z,\dq z\rangle\\
&= I_1^q+I_2^q+I_3^q.
\end{split}\end{equation*}
As in Section \ref{ss:integer}, we split $I_1^q$ and $I_2^q$ into
$$
I_1^q=2(I_{1,1}^q+ I_{1,2}^q+I_{1,3}^q+I_{1,4}^q+I_{1,5}^q) \qquad
\text{and}\qquad I_2^q=2(I_{2,1}^q+ I_{2,2}^q)
$$
where 
\begin{align*}
&I_{1,1} ^q= \int_{\R^N} \langle \dq z, \dq(-\frac{\sqrt{2}}{\eps}\nabla b)\rangle,
&&I_{2,1}^q = \int_{\R^N} \dq b\, \dq(-\frac{\sqrt{2}}{\eps}{\rm div}({\rm
Re}z)),\\
&I_{1,2}^q = \int_{\R^N} 
 b \langle\dq z, \dq(-\nabla b)\rangle,
&&I_{2,2}^q = \int_{\R^N}\dq b\,\dq(-{\rm div}(b{\rm Re}z)).\\
&I_{1,3}^q = \int_{\R^N} \langle\dq z, \dq(i\Delta z)\rangle,\\
&I_{1,4}^q = \int_{\R^N} \frac{\eps}{\sqrt{2}} b \langle \dq z, \dq(i\Delta z)\rangle,\\
&I_{1,5} ^q= \int_{\R^N} (1+\frac{\eps}{\sqrt{2}} b) \langle\dq z,\dq(-\nabla(\frac{z\cdot z}{2})\rangle,\\
\end{align*}
As in Section \ref{ss:integer}, both $I_{1,3}^q$ and
$I_{1,1}^q+I_{2,1}^q$ vanish. Next, in order to deal with 
$I_{1,2}^q+I_{2,2}^q,$ one may integrate by parts in $I_{2,2}^q.$ 
We find that
\begin{align*}
I_{1,2}^q + I_{2,2}^q &= \int_{\R^N} \langle \nabla\dq b,\dq(b{\rm Re}z) -
b\dq \Re z\rangle\\
&= \int_{\R^N}\!\dq b\bigl(b\,\dq\div\Re z -\dq (b\,\div\Re z)\bigr)  +
\int_{\R^N}\!\dq b\,\bigl(\nabla b\cdot\dq\Re z-\dq(\nabla b\cdot\Re z)\bigr)\\
&=  \int_{\R^N}  \dq b\,[b,\dq]\div\Re z+
\int_{\R^N} \dq b\,\nabla b\cdot\dq\Re z\\&\hspace{5cm}-
\int_{\R^N}\dq b\,\dq\nabla b\cdot\Re z
+\int_{\R^N}\dq b[\Re z,\dq]\cdot\nabla b
\\
&=  \int_{\R^N}  \dq b[b,\dq]\div\Re z +
\int_{\R^N}\dq b\,\nabla b\cdot\dq\Re z\\&\hspace{5cm}+\frac12
\int_{\R^N}(\dq b)^2\div\Re z
+\int_{\R^N}\dq b\,[\Re z,\dq]\cdot\nabla b.
\end{align*}
For the second and third term, we have
$$\begin{array}{lll}
\biggl|\Int_{\R^N} \dq b\,\nabla b\cdot\dq\Re z\biggr|
&\leq&\|Db\|_{L^\infty}\|\dq b\|_{L^2}\|\dq z\|_{L^2},\\[2ex]
\biggl|\Int_{\R^N} |\dq b|^2\, {\rm div}{\rm Re}z
\biggr| &\leq& \|\div z\|_{L^\infty} \|\dq b\|_{L^2}^2.\end{array}
$$
The first and last  terms may be bounded according to Lemma \ref{l:alamormoilnoeud}.
We find that for some sequence $(c_q)_{q\geq-1}$ such that 
$\|(c_q)\|_{\ell^r}=1,$
$$
\begin{array}{lll}
\left \vert  \Int_{\R^N} \dq b\,[b,\dq]\div\Re z  \ \right  \vert &\leq&
 Cc_q2^{-qs}\bigl(\|Db\|_{L^\infty}\|\div z\|_{B^{s-1}_{2,r}}
 +\|\div z\|_{L^\infty}\|Db\|_{B^{s-1}_{2,r}}\bigr)\|\dq b\|_{L^2}, \\[2ex]
\left \vert  \Int_{\R^N}\dq b\,[\Re z,\dq]\nabla b  \ \right  \vert &\leq&Cc_q2^{-qs}
 \bigl(\|Dz\|_{L^\infty}\|Db\|_{B^{s-1}_{2,r}}
 +\|Db\|_{L^\infty}\|Dz\|_{B^{s-1}_{2,r}}\bigr)\|\dq b\|_{L^2}.
\end{array}
$$
Combining the previous   inequalities,  we obtain
\begin{equation}\label{eq:besov2}
|I_{1,2}^q+I_{2,2}^q| \leq Cc_q2^{-qs}
 \bigl(\|Dz\|_{L^\infty}\|Db\|_{B^{s-1}_{2,r}}
 +\|Db\|_{L^\infty}\|Dz\|_{B^{s-1}_{2,r}}\bigr)\|\dq b\|_{L^2}.
 \end{equation}
To finish with, let us prove that
 \begin{equation}\label{eq:boutin6}
|2I_{1,4}^q+2I_{1,5}^q+I_3^q|
\leq Cc_q\bigl(1+\eps\|b\|_{L^\infty}\bigr) \|Dz\|_{L^\infty}
\|Dz\|_{B^{s-1}_{2,r}}\|\dq z\|_{L^2}.
\end{equation}
Integrating by
parts in $I_{1,4}^q,$  and using
the pointwise identity $\langle \dq\nabla z,\dq (i\nabla z)\rangle = 0$ and 
 \eqref{eq:pota},  we  derive   the identity 
$$
I_{1,4}^q = \int_{\R^N} (1+ \frac{\eps}{\sqrt{2}}b) \langle \dq z, \dq(i\nabla
z)\rangle {\rm Im}z. 
$$
Next, expanding $\nabla(\frac{z\cdot z}{2}),$ we
are led to
\begin{align*}
I_{1,5}^q &= -\int_{\R^N} (1+\frac{\eps}{\sqrt{2}} b) \langle \dq z,
\dq(\nabla z)\cdot z\rangle 
+ \int_{\R^N} (1+\frac{\eps}{\sqrt{2}} b) \langle \dq z,
\dq\nabla z\cdot z-\dq(z\cdot\nabla z)\rangle\\
&= {I'}^q_{1,5} + {I''}^q_{1,5}.
\end{align*} 
On the one hand,
Lemma \ref{l:alamormoilnoeud} ensures that ${I''}^q_{1,5}$ may
be bounded by the right-hand side of \eqref{eq:boutin6}. 
On the other hand,  mimicking the computations made in Section
\ref{ss:integer}, we get
\begin{align*}
2\bigl({I'}^q_{1,5}+I^q_{1,4}\bigr)+I^q_3 &= - \int_{\R^N}
(1+\frac{\eps}{\sqrt{2}}b)\Re z\cdot \nabla |\dq z|^2+  \int_{\R^N}
\frac{\eps}{\sqrt{2}}\partial_t b \langle \dq z,\dq z\rangle \\ 
&= \int_{\R^N} {\rm div}\biggl(\Bigl(1+\frac{\eps}{\sqrt{2}}b\Bigr) \Re  z\biggr)\,|\dq z|^2
+  \int_{\R^N}\frac{\eps}{\sqrt{2}}\partial_t b\,|\dq z|^2
\end{align*}
so that, since system \eqref{eq:seven} is satisfied, 
$$
2{I'}^q_{1,5}+2I^q_{1,4}+I^q_3=0.
$$
This completes the proof of \eqref{eq:boutin6}, 
and thus of \eqref{eq:besov1}.
\end{proof}


\section{Proof of Theorems \ref{thm:un} and \ref{thm:deux}}

We first notice that by Sobolev embedding and the definition of $a_\eps$ there exists
	 a constant $C_1(s,N)\geq 1$ independent of
	 $\eps$ such that if $C_1(s,N)\eps \|a_\eps\|_{H^{s+1}} \leq 1$ then 
	 \begin{equation}\label{eq:novortex1}
	 \rho_\eps^{\pm1}\leq2.
	 \end{equation}
	 	The constant $C(s,N)$ will be required to satisfy $C(s,N) > C_1(s,N),$ so that in particular $\rho_\eps^{\pm1}(\cdot,0)<{2}.$ If we denote by $\Psi^0$ the corresponding initial 
datum for $(GP)$ then one may prove that $\Psi^0\in\mathcal{V}+H^{s+1}(\R^N).$ 
In fact, it turns out that for any 
 smooth nonnegative function $\alpha$ compactly supported in $\R^N$ and satisfying
$\int \alpha =1$ the function  $U:= \Psi * \alpha$
 belongs to $\mathcal{V}$ and $\Psi^0-U$
belongs to $H^{s+1}(\R^N)$ (see e.g. \cite{Gal}).  
Therefore, by virtue of  Proposition \ref{prop:BS}, equation $(GP)$ 
possesses a unique solution $\Psi$ in $\Psi^0 + H^{s+1}$ on some time interval $[0,T]$. 
\medbreak

\noindent{\bf Proof of  Theorem \ref{thm:un}}

\noindent{\bf Step 1: } 	
	In a first step, we assume that in addition 
	$(a_\eps^0,u_\eps^0) \in H^{s+3}\times H^{s+2}$. By 
	Proposition \ref{prop:BS} {\it iii)} combined with an appropriate 
	change of variable, equation \eqref{eq:habit} has 
	 a unique maximal  solution $\Upsilon_\eps$
 in $\mathcal V+\mathcal{C}^1([0,T^s);H^{s+1}(\R^N)).$  We introduce the stopping time
$$
t_0 = {\rm Sup}\left\{ 0\leq t<T^s \quad\text{s.t.}\quad 
\rho_\eps^{\pm1} \leq{2}\quad \text{and}\quad A(\tau) \leq 2 A(0),\quad  \forall \tau \in
[0,t]\right\},
$$
where we have set\footnote{For expository purposes, 
we assume here that $s$ is an integer number.}
$$
A(t) :=  \Gamma^s(b(\cdot,t),z(\cdot,t))+E_\eps(\Upsilon_\eps(\cdot, t)).
$$

By continuity and the fact that $\rho(\cdot,0) >\frac 1 2,$ we have $t_0>0.$
Next, we apply Proposition \ref{prop:1}
 on the interval $[0,t_0)$, which yields the
inequality
$$
\frac{d}{dt}\Gamma^s(b,z) \leq C_2(s,N) \|(Db,Dz)\|_{L^\infty}\left(
\Gamma^s(b,z)+E_\eps(\Upsilon_\eps)\right).
$$
On the one hand, by conservation of energy, we have on $[0,T)$,
$$
\frac{d}{dt}E_\eps(\Upsilon_\eps) = 0.
$$
On the other hand, by Sobolev embedding and \eqref{eq:victor0}, we have
$$
\|(Db,Dz)\|_{L^\infty}^2 \leq C_3(s,N)\left( \Gamma^s(b,z)+E_\eps(\Upsilon_\eps)\right).
$$
Therefore, after summation
we are led to
$$
\frac{d}{dt} A(t) \leq C_4(s,N) A(t)^{3/2}\qquad\text{on }\ [0,t_0).
$$
Integrating this inequality we obtain   
$$
A(t) \leq \frac{A(0)}{(1-C_4(s,N)\sqrt{A(0)}t/2)^2}\quad\hbox{whenever }\ t<t_1:=\min(t_0, \frac{2}{\sqrt{A(0)C_4}}).
$$
Notice that, owing to \eqref{eq:novortex1} and to the definition of $\Gamma^s,$ we have
\begin{equation}\label{eq:novortex2}
\frac12\|(a_\eps^0,u_\eps^0)\|_{H^{s+1}\times H^s}\leq
\sqrt{A(0)}\leq 2  \|(a_\eps^0,u_\eps^0)\|_{H^{s+1}\times H^s}.
\end{equation}
Therefore, choosing $C(s,N)$ sufficiently large, we have, for $t\leq t_*:=
\frac{1}{C(s,N)\|a_\eps^0,u_\eps^0\|_{H^{s+1}\times H^s}},$
$$
C_4(s,N) \sqrt{A(0)} t/2 \leq C_5(s,N) \|(a_\eps^0,u_\eps^0)\|_{H^{s+1}\times H^s}
\cdot \frac{1}{C(s,N)\|(a_\eps^0,u_\eps^0)\|_{H^{s+1}\times H^s}} \leq \frac{1}{2}, 
$$
so that 
$$
A(t) \leq 2 A(0)\quad\hbox{whenever }\ t\leq \min(t_0,t_*). 
$$
For such $t$, we then have
$$
\eps \|(a_\eps(\cdot,\frac t \eps),u_\eps(\cdot, \frac t \eps)\|_{H^{s+1}\times H^s}
 \leq C_6(s,N) \eps \|(a_\eps^0,u_\eps^0\|_{H^{s+1}\times H^s}
  \leq \frac{C_6(s,N)}{C(s,N)}
$$
so that
condition \eqref{eq:novortex1} is satisfied  provided $C(s,N)$ is chosen sufficiently large. 
It follows that $t_0>t_*$. The case where $s\not\in\N$  follows from the same arguments. It suffices to apply Proposition \ref{prop:1besov} with $r=2$
instead of Proposition \ref{prop:1}. The conclusion in Theorem \ref{thm:un} therefore holds in the case considered in this step.

\medbreak
\noindent{\bf Step 2:} The general case.  In order to prove Theorem \ref{thm:un} in the general case, we  mollify the inital datum by an approximation of the identity and then  rely on Case 1 and the continuity of the flow map on $\mathcal{C}^0([0,T];\Psi^0+H^{s+1}).$ 
The details are standard and left to the reader. \qed  

\medbreak\noindent{\bf Proof of Theorem \ref{thm:deux}}

We notice that $(\tilde a_\eps,\tilde u_\eps):=(a_\eps,u_\eps)-(\mathfrak a,\mathfrak u)$
satisfies the wave equation
$$
\left\{
\begin{array}{l}
\displaystyle
\partial_t \tilde a_\eps + \sqrt{2}\, {\rm div}\,\tilde u_\eps = \eps\,\div f^1_\eps,\\
\displaystyle
\partial_t \tilde u_\eps + \sqrt{2}\, \nabla\tilde a_\eps = \eps\,\nabla\bigl(g^1_\eps
+2 g^2_\eps\bigr)
\end{array}
\right.
$$ 
with null initial datum and
$$
f^1_\eps:=-a_\eps u_\eps,\quad
g^1_\eps:=-|u_\eps|^2\ \hbox{ and }\ 
g^2_\eps:=\frac{\Delta\sqrt{\sqrt 2+\eps a_\eps}}
{\sqrt{\sqrt 2+\eps a_\eps}}\cdotp
$$
Using basic energy estimates for the wave equation, we readily get
\begin{equation}\label{eq:diff0}
\|(\tilde a_\eps,\tilde u_\eps)(t)\|_{H^{s-2}}
\leq \int_0^t\bigl(\|f^1_\eps\|_{H^{s-1}}
+\|g^1_\eps\|_{H^{s-1}}+2\|g^2_\eps\|_{H^{s-1}}\bigr)\,d\tau.
\end{equation}
Now, as $H^{s-1}$ is an algebra, one can write
$$
\|f^1_\eps\|_{H^{s-1}}\leq C\|a_\eps\|_{H^{s-1}}\|u_\eps\|_{H^{s-1}}\ \hbox{ and }\
\|g^1_\eps\|_{H^{s-1}}\leq C\|u_\eps\|_{H^{s-1}}^2,
$$
whence, according to Theorem \ref{thm:un}, 
\begin{equation}\label{eq:diff1}
\|f^1_\eps\|_{H^{s-1}}
+\|g^1_\eps\|_{H^{s-1}}\leq 
C\|(a_\eps^0,u_\eps^0)\|^2_{H^{s+1}\times H^s}\quad\hbox{for all }\ 
t\in[0,T_\eps).
\end{equation}
In order to bound the last term in \eqref{eq:diff0}, we notice that,
under condition \eqref{eq:novortex1}, 
there exist two smooth functions $K_1$ and $K_2$ vanishing at $0$ and such that
$$
g_\eps^2=\frac\eps4\Delta a_\eps+\eps K_1(\eps a_\eps)\Delta a_\eps
+\eps \nabla a_\eps\cdot\nabla K_2(\eps a_\eps).
$$
Therefore, applying Proposition \ref{p:compo} yields
$$
\|g_\eps^2\|_{H^{s-1}}\leq C\bigl(\eps\|a_\eps\|_{H^{s+1}}
+\eps^2\|a_\eps\|_{H^{s-1}}\|a_\eps\|_{H^{s+1}}
+\eps^2\|a_\eps\|_{H^s}^2\bigr),
$$
so that, using the bounds provided by Theorem \ref{thm:un}, 
\begin{equation}\label{eq:diff2}
\|g^2_\eps\|_{H^{s-1}}
\leq  C\bigl(\eps\|(a_\eps^0,u_\eps^0)\|_{H^{s+1}\times H^s}
+\eps^2\|(a_\eps^0,u_\eps^0)\|_{H^{s+1}\times H^s}^2\bigr)
\quad\hbox{for all }\ 
t\in[0,T_\eps).
\end{equation}
Using inequalities \eqref{eq:diff1} and \eqref{eq:diff2} 
in \eqref{eq:diff0}, it is now easy to complete
the proof of Theorem~\ref{thm:deux}. \qed

\begin{rem}
Using Proposition \ref{prop:1besov}, the results in Theorem \ref{thm:un} and Theorem \ref{thm:deux}  can be extended to the spaces $B^{s+1}_{2,r}\times B^s_{2,r}$ instead of $H^{s+1}\times H^s$   whenever $B^s_{2,r}$ embeds continuously in 
$W^{1,\infty}.$ The Besov spaces framework
 allows to get a result for the critical regularity $s=1+N/2$ since 
 $B^{1+N/2}_{2,1}\hookrightarrow W^{1,\infty}$ (whereas $H^{1+N/2}\not\hookrightarrow
 W^{1,\infty}$). 
   \end{rem}

\section{Proof of Theorems \ref{thm:trois} and  \ref{thm:quatre}}\label{prop:5}

As mentioned in the introduction, our proofs will be based on the dispersive
properties of the linearized 
 system \eqref{eq:dyna} about $(0,0),$ namely
 \begin{equation}\label{eq:linearized}
\left\{
\begin{array}{l}
\displaystyle
\partial_t b + \frac{\sqrt{2}}{\eps}{\rm div}v = f,
\\
\displaystyle
\partial_t v + \frac{\sqrt{2}}{\eps} \nabla b-\sqrt2\eps\nabla\Delta b = g.
\end{array}
\right.
\end{equation}
More precisely, we shall use the following 
result, the proof of which is presented in the Appendix.
\begin{prop}\label{p:dispersive}
Let $(b,v)$ solve system $\eqref{eq:linearized}$
on $[0,T]\times\R^N.$
There exists a smooth compactly supported function $\chi$
with  value $1$ near the origin such that for all $\eps>0,$
if we denote  
\begin{equation}\label{eq:def}
b_\ell:=\chi(\eps D)b,\quad
b_h:=(1-\chi(\eps D)) b,\quad
v_\ell:=\chi(\eps D)v\quad\!\!\hbox{and}\!\!\quad
v_h:=(1-\chi(\eps D)) v
\end{equation}
then the  following a priori estimates hold true for some constant
$C$ depending only on $N$:
\begin{itemize}
\item if $N\geq4$ then
$\|(b,v)_{\ell}\|_{L^2_T(C^{0,1})}\leq C\eps^{\frac12}
\Bigl(\|(b_0,v_0)_{\ell}\|_{B^{\frac N2+\frac12}_{2,1}}+
\|(f,g)_{\ell}\|_{L^1_T(B^{\frac N2+\frac12}_{2,1})}\Bigr)$
\item if $N=3$ then for all $p>2,$
$\|(b,v)_{\ell}\|_{L^p_T(C^{0,1})}\! \leq\! C\eps^{\frac1p}
\Bigl(\|(b_0,v_0)_{\ell}\|_{B^{\frac 52\!-\!\frac1p}_{2,1}}\!+\!
\|(f,g)_{\ell}\|_{L^1_T(B^{\frac 52\!-\!\frac1p}_{2,1})}\Bigr)$
\item if $N=2$ then 
$\|(b,v)_{\ell}\|_{L^4_T(C^{0,1})}\leq C\eps^{\frac14}
\Bigl(\|(b_0,v_0)_{\ell}\|_{B^{\frac74}_{2,1}}+
\|(f,g)_{\ell}\|_{L^1_T(B^{\frac74}_{2,1})}\Bigr)$
\item if $N\geq3$ then 
$\|(\eps\nabla b,v)_{h}\|_{L^2_T(C^{0,1})}\leq C
\Bigl(\|(\eps\nabla b_0,v_0)_{h}\|_{B^{\frac N2}_{2,1}}+
\|(\eps\nabla f,g)_{h}\|_{L^1_T(B^{\frac N2}_{2,1})}\Bigr)$
\item if $N=2$ then for all $p>2,$
$$\|(\eps\nabla b,v)_h\|_{L^p_T(C^{0,1})}\leq C
\Bigl(\|(\eps\nabla b_0,v_0)_{h}\|_{B^{2-\frac2p}_{2,1}}+
\|(\eps\nabla f,g)_{h}\|_{L^1_T(B^{2-\frac2p}_{2,1})}\Bigr).$$
\end{itemize}\end{prop}

Throughout the proof of Theorem \ref{thm:trois}, we shall use freely the following 
inequalities which are proved in the Appendix:
\begin{lem} With the notation used in Proposition \ref{p:dispersive},   there exists a constant $C>0$ depending only on $N$ and $\sigma>0$ such that, under
condition \eqref{eq:novortex1},
\label{l:rouge}
\begin{eqnarray}\label{eq:equiv1}
&\!\!\!C^{-1}\|(b,z)\|_{B^\sigma_{2,r}}\leq \|(b,v)_\ell\|_{B^\sigma_{2,r}}
\!+\!\|(\eps\nabla b,v)_h\|_{B^\sigma_{2,r}}\leq C\|(b,z)\|_{B^\sigma_{2,r}}
\quad\!\!\hbox{for } \ r\in[1,\infty],\quad\\[2ex]
\label{eq:equiv2}
&C^{-1}\|(b,z)\|_{C^{0,1}}\leq \|(b,v)_\ell\|_{C^{0,1}}
+\|(\eps\nabla b,v)_h\|_{C^{0,1}}\leq C\|(b,z)\|_{C^{0,1}}.
\end{eqnarray}
\end{lem}

\smallskip

\subsection{Proof of Theorem \ref{thm:trois} in the case 
\mathversion{bold}$N\geq 4$\mathversion{normal}}

According to Proposition \ref{p:dispersive}, the linear system
\eqref{eq:linearized} possesses better dispersive
properties in high dimension $N\geq4.$ 
Therefore, we shall first prove Theorem \ref{thm:trois}
in this case. 
\smallbreak
Assume that we are given some map
$\Psi$ solution of $(GP)$ with 
datum $\Psi^0,$  satisfying $(b,z)\in{\mathcal C}^1([0,T];H^s)$ and 
\begin{equation}\label{eq:strichartz0}
\frac12\leq\rho\leq 2\quad\hbox{on}\quad[0,T].
\end{equation}
Integrating  the inequality in Proposition \ref{prop:1besov} in the case $r=2$ 
and taking  inequality \eqref{eq:strichartz0} into account
 yields  for all $t\in[0,T],$ 
$$
\|(b,z)(t)\|_{H^s}
\leq 2\|(b_0,z_0)\|_{H^s}+C\int_0^t\|(Db,Dz)\|_{L^\infty}
\|(b,z)\|_{H^s}\,d\tau,
$$
whence, according to Cauchy-Schwarz inequality and to inequality \eqref{eq:equiv2},
\begin{equation}\label{eq:strichartz1}
\!\!\!\|(b,z)\|_{L^\infty_t(H^s)}
\!\leq\!2\|(b_0,z_0)\|_{H^s}\!+\!Ct^{\frac12}\bigl(\|(b,v)_{\!\ell}\|_{L_t^2(C^{0,1})}
\!+\!\|(\eps\nabla b,v)_{\!h}\|_{L_t^2(C^{0,1})})
\|(b,z)\|_{L^\infty_t(H^s)}.\!
\end{equation}
In order to bound  $(b,v)_\ell$ and $(\eps\nabla b,v)_h$ in $L^2([0,T];C^{0,1}),$
we shall  take advantage of Proposition 
\ref{p:dispersive}. As $N\geq4$ and  
 $(b,v)$ satisfies system \eqref{eq:linearized} with source 
terms\footnote{To get the formula for $g_2,$ it suffices to use
\eqref{eq:dyna} and the identities \eqref{eq:rho}, \eqref{eq:pota}
which imply that $$
\frac{\Delta\rho}\rho=-\div\Im z\quad\hbox{and}\quad
\frac{\eps}{\sqrt2}\,\Delta b=-\frac\eps{\sqrt2}\,\div(b\,\Im z)-\div\Im z.
$$}
 $$ f:=-\div(bv)\quad\hbox{and}\quad g:=g_1+g_2\ \hbox{ with }\ g_1:=-\nabla
|v|^2\ \hbox{ and }\ g_2:=\sqrt2\,\eps\nabla\div(b\,\Im z),$$
 we get for all $t\in[0,T],$
\begin{eqnarray}\label{eq:strichartz2a}
\|(b,v)_\ell\|_{L_t^2(C^{0,1})}
\leq C\eps^{\frac12}
\bigl(\|(b_0,v_0)_\ell\|_{B^{\frac N2+\frac12}_{2,1}}
+\|(f,g)_\ell\|_{L_t^1(B^{\frac N2+\frac12}_{2,1})}\bigr),\\
\label{eq:strichartz2b}
\|(\eps Db,v)_h\|_{L_t^2(C^{0,1})}
\leq C\bigl(\|(\eps Db_0,v_0)_h\|_{B^{\frac N2}_{2,1}}
+\|(\eps Df,g)_h\|_{L_t^1(B^{\frac N2}_{2,1})}\bigr).
\end{eqnarray}
We claim that 
\begin{equation}\label{eq:strichartz2c}
\eps^{\frac12}\|(b_0,v_0)_\ell\|_{B^{\frac N2+\frac12}_{2,1}}
+\|(\eps Db_0,v_0)_h\|_{B^{\frac N2}_{2,1}}
\leq C\eps^{\frac12}\|(b_0,z_0)\|_{B^{\frac N2+\frac12}_{2,1}}.
\end{equation}
In fact,
for  the high frequency part of the datum, one may exchange
 the factor $\eps^{\frac12}$ against half a derivative. 
This is due to the fact that for all $\sigma\in\R,$ $\alpha>0$ and $\phi\in
B^{s+\alpha}_{2,1},$ we have \begin{equation}\label{eq:strichartzHF}
\|\phi_h\|_{B^s_{2,1}}\leq C\eps^{\alpha}
\|\phi_h\|_{B^{s+\alpha}_{2,1}}.
\end{equation}
 Indeed, owing to the support properties
of function $\chi,$ one may write for some suitable $\eps_0>0,$
 $$
\|\phi_h\|_{B^s_{2,1}}=\sum_{2^q\eps\geq\eps_0}
2^{qs}\|\Delta_q\phi_h\|_{L^2} \leq
\Bigl(\frac\eps{\eps_0}\Bigr)^{\alpha}\!\sum_{2^q\eps\geq\eps_0}
2^{q(s+\alpha)}\|\Delta_q\phi_h\|_{L^2}\leq\Bigl(\frac\eps{\eps_0}\Bigr)^{\alpha}
\|\phi_h\|_{B^{s+\alpha}_{2,1}}.
$$
Therefore 
$$
\eps^{\frac12}\|(b_0,v_0)_\ell\|_{B^{\frac N2+\frac12}_{2,1}}
+\|(\eps Db_0,v_0)_h\|_{B^{\frac N2}_{2,1}}
\leq C\eps^{\frac12}\bigl(\|(b_0,v_0)_\ell\|_{B^{\frac N2+\frac12}_{2,1}}
+\|(\eps Db_0,v_0)_h\|_{B^{\frac N2+\frac12}_{2,1}}\bigr)
$$
and applying inequality \eqref{eq:equiv1}
gives \eqref{eq:strichartz2c}.
\smallbreak
It follows from the previous discussion  that  the problem reduces to finding  suitable bounds 
for $(f,g)_\ell$ in $L^1([0,T];B^{\frac N2+\frac12}_{2,1}),$
and for $(\eps\nabla f,g)_h$ in $L^1([0,T];B^{\frac N2}_{2,1}).$
For that purpose, we  use standard tame estimates for the product
of functions in Besov spaces  which are stated in  Proposition \ref{p:besov}.
This yields
$$
\|f_\ell\|_{B^{\frac N2+\frac12}_{2,1}}
\leq C\|bv\|_{B^{\frac N2+\frac32}_{2,1}}
\leq C\bigl(\|b\|_{L^\infty}\|v\|_{B^{\frac N2+\frac32}_{2,1}}
+\|v\|_{L^\infty}\|b\|_{B^{\frac N2+\frac32}_{2,1}}\bigr),
$$
$$
\|(g_1)_\ell\|_{B^{\frac N2+\frac12}_{2,1}}
\leq C\||v|^2\|_{B^{\frac N2+\frac32}_{2,1}}
\leq C\|v\|_{L^\infty}\|v\|_{B^{\frac N2+\frac32}_{2,1}}.
$$
To deal with the term $(g_2)_\ell,$ 
we notice that for all $\sigma\in\R$ and $\phi\in B^{\sigma}_{2,1},$ we have
\begin{equation}\label{eq:strichartzLF}
\|\eps\nabla\phi_\ell\|_{B^\sigma_{2,1}}\leq
C\|\phi_\ell\|_{B^{\sigma}_{2,1}}. \end{equation}
Indeed, owing to the support properties of $\Supp\hat\phi_\ell$
and Parseval formula, we have for some $\eps_1>\eps_0,$
   $$
\|\eps\nabla \phi_\ell\|_{B^\sigma_{2,1}}=\sum_{2^q\eps\leq\eps_1}
2^{q\sigma}\eps\|\nabla\Delta_q\phi_\ell\|_{L^2} \leq
\sum_{2^q\eps\leq\eps_1}
(\eps 2^q)
2^{q\sigma}\|\Delta_q\phi_\ell\|_{L^2}\leq\eps_1
\|\phi_\ell\|_{B^{\sigma}_{2,1}}. $$ 
Therefore,
$$ \begin{array}{lll}
\|(g_2)_\ell\|_{B^{\frac N2+\frac12}_{2,1}}
&\leq& C\|b\,\Im z\|_{B^{\frac N2+\frac32}_{2,1}},\\
&\leq& C\bigl(\|b\|_{L^\infty}\|\Im z\|_{B^{\frac N2+\frac32}_{2,1}}
+\|\Im z\|_{L^\infty}\|b\|_{B^{\frac N2+\frac32}_{2,1}}\bigr).\end{array}
$$
Summing the    inequalities above, we end up with
\begin{equation}\label{eq:strichartz3}
\|(f,g)_\ell\|_{B^{\frac N2+\frac12}_{2,1}}\leq C\|(b,z)\|_{L^\infty}
\|(b,z)\|_{B^{\frac N2+\frac32}_{2,1}}.
\end{equation}
To deal with the high-frequency terms,  we  use Proposition 
\ref{p:besov} once more. We obtain
\begin{eqnarray}\label{eq:strichartz3a}
&\|\eps\nabla f_h\|_{B^{\frac N2}_{2,1}}
\leq C\eps\|bv\|_{B^{\frac N2+2}_{2,1}}
\leq C\eps\bigl(\|b\|_{L^\infty}\|v\|_{B^{\frac N2+2}_{2,1}}
+\|v\|_{L^\infty}\|b\|_{B^{\frac N2+2}_{2,1}}\bigr),\\\label{eq:strichartz3b}
&\|(g_2)_h\|_{B^{\frac N2}_{2,1}}
\leq C\eps\|b\,\Im z\|_{B^{\frac N2+2}_{2,1}}
\leq C\eps\bigl(\|b\|_{L^\infty}\|\Im z\|_{B^{\frac N2+2}_{2,1}}
+\|\Im z\|_{L^\infty}\|b\|_{B^{\frac N2+2}_{2,1}}\bigr).
\end{eqnarray}
A direct  estimate of $(g_1)_h$ would give a term  of order $1.$
To get the factor $\eps,$ one
may first take advantage of  inequality \eqref{eq:strichartzHF}
so as to get, 
$$\begin{array}{lll}
\|(g_1)_h\|_{B^{\frac N2}_{2,1}}
&\leq& C\eps\|g_1\|_{B^{\frac N2+1}_{2,1}},\\
&\leq& C\eps\||v|^2\|_{B^{\frac N2+2}_{2,1}},\\
&\leq&  C\eps
\|v\|_{L^\infty}\|v\|_{B^{\frac N2+2}_{2,1}}.
\end{array}
$$
The above  inequality together with \eqref{eq:strichartz3a} and
\eqref{eq:strichartz3b} implies that  \begin{equation}\label{eq:strichartz4}
\|(\eps\nabla f,g)_h\|_{B^{\frac N2+\frac12}_{2,1}}
\leq C\eps\|(b,z)\|_{L^\infty}
\|(b,z)\|_{B^{\frac N2+2}_{2,1}}.
\end{equation}
Finally, inserting inequalities \eqref{eq:strichartz2c},
\eqref{eq:strichartz3},  \eqref{eq:strichartz4} into  \eqref{eq:strichartz2a},
\eqref{eq:strichartz2b}, we  end up with
$$\displaylines{
\|(b,v)_\ell\|_{L^2_t(C^{0,1})}+\|(\eps\nabla b,v)_h\|_{L^2_t(C^{0,1})}
\leq C\eps^{\frac12}\Bigl(\|(b_0,z_0)\|_{B^{\frac N2+\frac12}_{2,1}}
\hfill\cr\hfill+t^{\frac12}\|(b,z)\|_{L_t^2(L^\infty)}
\|(b,z)\|_{L_t^\infty(B^{\frac N2+2}_{2,1})}\Bigr).}
$$
Since $s>N/2+2,$ we have  $H^s\hookrightarrow B^{\frac
N2+\frac12}_{2,1}\hookrightarrow B^{\frac
N2+2}_{2,1}$ so that 
$$
\|(b,v)_\ell\|_{L^2_t(C^{0,1})}+\|(\eps\nabla b,v)_h\|_{L^2_t(C^{0,1})}
\leq C\eps^{\frac12}\Bigl(\|(b_0,z_0)\|_{H^s}
+t^{\frac12}\|(b,z)\|_{L_t^2(L^\infty)}
\|(b,z)\|_{L_t^\infty(H^s)}\Bigr).
$$
Let $X_0:=\|(b_0,z_0)\|_{H^s}$ and $X(t):= \|(b,z)\|_{L^\infty_t(H^s)}
+c\eps^{-\frac12}\|(b,z)\|_{L^2_t(C^{0,1})}$  where $c=c(S, N)$ is some constant which is assumed to be  sufficiently small. We deduce from the previous inequality,
 \eqref{eq:strichartz1} and Lemma \ref{l:rouge},
 that,  changing possibly the  constant $C,$
\begin{equation}\label{eq:boot}
X(t)\leq 3X_0+C\sqrt{\eps t}\, X^2(t).
\end{equation}
Therefore, using a stopping time  argument
similar to  the  one used in  the proof of Theorem \ref{thm:un}, 
we  conclude that $X(t)\leq 4X(0)$ for all $t\in[0,T]$ 
whenever $T$ satisfies
$$
16CX_0\sqrt{\eps T}\leq 1.
$$
We finally  complete the proof of  Theorem \ref{thm:trois} in the case $N\geq 4$ as the end of the proof of Theorem \ref{thm:un}. The
details are left to the reader.
\qed
\begin{rem}
The above proof
may be easily adapted to the Besov spaces framework, and in 
particular to the case where the data $(a_0^\eps,u_0^\eps)$ are in 
$B^{\frac N2+3}_{2,1}\times
B^{\frac N2+2}_{2,1}$ and satisfy $$C\eps\|(a_\eps^0,u_\eps^0)\|_{B^{\frac N2+3}_{2,1}\times
B^{\frac N2+2}_{2,1}}\leq 1.$$  As an easy consequence, 
we discover that  under the conditions of Theorem $\ref{thm:trois}$
 there exists a constant $c$ independent of $s$
 such that $|\Psi|$ remains bounded away from zero
 up to time  
 $$\frac c{\eps^2\|(a_\eps^0,u_\eps^0)\|_{B^{\frac N2+3}_{2,1}\times B^{\frac N2+2}_{2,1}}^2}\cdotp$$ 
 \end{rem}

 
\subsection{Proof of Theorem \ref{thm:trois} in the case 
\mathversion{bold}$N=3$\mathversion{normal}}

 The proof   Theorem \ref{thm:trois} in the three-dimensional case relies on very  similar arguments:   however, 
the endpoint inequality pertaining to $p=2$ 
in Proposition \ref{p:dispersive} does not hold for $N=3$ 
and has to be replaced by slightly more technical arguments. 

As above, we assume that we are given some suitably smooth map
$\Psi,$ solution of $(GP)$ with 
datum $\Psi^0$ and satisfying \eqref{eq:strichartz0}. 
Fix some $\alpha\in(0,1),$ and set 
 $p:=1+1/\alpha$ and $p'=\alpha+1.$
Arguing as  for   the proof  \eqref{eq:strichartz1}, we obtain
\begin{eqnarray}\label{eq:N=3:1}
&&\|(b,z)\|_{L^\infty_t(H^s)}
\leq2\|(b_0,z_0)\|_{H^s}\nonumber\\
&&\hspace{3cm}+C\Bigl(t^{\frac1{p'}}\|(b,v)_{\ell}\|_{L_t^p(C^{0,1})}
+t^{\frac12}\|(\eps\nabla b,v)_{h}\|_{L_t^2(C^{0,1})}\Bigr)
\|(b,z)\|_{L^\infty_t(H^s)}.\qquad
\end{eqnarray}
It remains to find appropriate   bounds for $(b,v)_{\ell}$ in $L^p([0,T];C^{0,1})$
and $(\eps\nabla b,v)_{h}$ in $L^2([0,T];C^{0,1}).$
For the low-frequency part of the solution, 
 Proposition \ref{p:dispersive} ensures that
\begin{equation}\label{eq:N=3:2}
\|(b,v)_\ell\|_{L_t^p(C^{0,1})}
\leq C\eps^{\frac1p}
\bigl(\|(b_0,v_0)_\ell\|_{B^{\frac52-\frac1p}_{2,1}}
+\|(f,g)_\ell\|_{L_t^1(B^{\frac52-\frac1p}_{2,1})}\bigr).
\end{equation}
As in the case $N\geq4,$ 
the source terms $f_\ell$ and $g_\ell$ may be easily bounded  thanks to
Proposition \ref{p:besov}. We end up with 
\begin{equation}\label{eq:N=3:3}
\|(f,g)_\ell\|_{B^{\frac52-\frac1p}_{2,1}}\leq C\|(b,z)\|_{L^\infty}
\|(b,z)\|_{B^{\frac72-\frac1p}_{2,1}}.
\end{equation}
To deal with the high-frequency terms, we notice that, by virtue of Proposition
\ref{p:dispersive}, we have
\begin{equation}\label{eq:N=3:4}
\|(\eps Db,v)_h\|_{L_t^2(C^{0,1})}
\leq C\bigl(\|(\eps Db_0,v_0)_h\|_{B^{\frac32}_{2,1}}
+\|(\eps Df,g)_h\|_{L_t^1(B^{\frac32}_{2,1})}\bigr).
\end{equation}
Taking advantage of Proposition \ref{p:besov} and inequality
\eqref{eq:strichartzHF},   we get $$
\begin{array}{lll}
\|\eps\nabla f_h\|_{B^{\frac32}_{2,1}}
&\leq& C\eps\bigl(\|b\|_{L^\infty}\|v\|_{B^{\frac72}_{2,1}}
+\|v\|_{L^\infty}\|b\|_{B^{\frac72}_{2,1}}\bigr),\\[2ex]
\|(g_1)_h\|_{B^{\frac32}_{2,1}}&\leq& C\eps\|(g_1)_h\|_{B^{\frac52}_{2,1}}
\leq C\eps\|v\|_{L^\infty}\|v\|_{B^{\frac72}_{2,1}},\\[2ex]
\|(g_2)_h\|_{B^{\frac32}_{2,1}}
&\leq&C\eps\bigl(\|b\|_{L^\infty}\|\Im z\|_{B^{\frac72}_{2,1}}
+\|\Im z\|_{L^\infty}\|b\|_{B^{\frac72}_{2,1}}\bigr).
\end{array}
$$
 Following  the lines of the computations leading to 
inequality \eqref{eq:strichartz2c}, it is not difficult to show that
$$
\|(\eps Db_0,v_0)_h\|_{B^{\frac32}_{2,1}}
\leq C\eps\|(b_0,z_0)\|_{B^{\frac52}_{2,1}}.
$$
It follows that, if $s>7/2$ then inequalities \eqref{eq:N=3:2} to \eqref{eq:N=3:4}
yield $$\displaylines{
\eps^{-\frac1p}\|(b,v)_\ell\|_{L_t^p(C^{0,1})}
+\eps^{-1}\|(\eps Db,v)_h\|_{L_t^2(C^{0,1})}
\leq C\Bigl(\|(b_0,v_0)\|_{H^s}\hfill\cr\hfill+\Bigl(t^{\frac1{p'}}
\|(b,v)_{\ell}\|_{L_t^p(C^{0,1})} +t^{\frac12}\|(\eps\nabla
b,v)_{h}\|_{L_t^2(C^{0,1})}\Bigr)\|(b,z)\|_{L^\infty_t(H^s)}\Bigr).}
$$
We introduce, for a  constant $c$ which is assumed to be arbitrarily small, the quantity 
$$
X(t):=\|(b,z)\|_{L^\infty_t(H^s)}+c\eps^{-\frac1p}\|(b,v)_\ell\|_{L_t^p(C^{0,1})}
+c\eps^{-1}\|(\eps Db,v)_h\|_{L_t^2(C^{0,1})}.
$$
We obtain, in view of the prevous estimates
$$
X(t)\leq 3X_0+C\bigl(\eps^{\frac1p}t^{\frac1{p'}}+\eps\sqrt t\bigr)X^2.
$$
Using a standard bootstrap argument, 
one can conclude that $X(t)\leq 4X(0)$ for all $t\in[0,T]$ 
whenever $T$ satisfies
$$
16C\bigl(\eps^{\frac1p}t^{\frac1{p'}}+\eps\sqrt t\bigr)X_0\leq 1.
$$
It is  then straightforward to  complete the proof of the theorem.
\qed
\begin{rem}
As a by-product of the above proof, we see that if
 the data $(a_0^\eps,u_0^\eps)$ are in 
$B^{\frac92}_{2,1}(\R^3)\times B^{\frac72}_{2,1}(\R^3)$ 
and satisfy $$C\eps\|(a_\eps^0,u_\eps^0)\|_{B^{\frac92}_{2,1}\times
B^{\frac72}_{2,1}}\leq 1$$ then for all $\alpha\in(0,1),$ 
 there exists a constant $c$ 
 such that $|\Psi|$ remains bounded away from zero
 up to time  
 $$\min\biggl(\frac c
 {\eps^{1+\alpha}\|(a_\eps^0,u_\eps^0)\|_{B^{\frac92}_{2,1}\times B^{\frac72}_{2,1}}^{1+\alpha}},
 \frac c
 {\eps^{3}\|(a_\eps^0,u_\eps^0)\|_{B^{\frac92}_{2,1}\times B^{\frac72}_{2,1}}^{2}}
\biggr) \cdotp$$ 
Let us also point out that, resorting
to the logarithmic Strichartz estimate 
(as in e.g. \cite{BCD}, Th. 8.27), 
one may replace the factor $\eps^{1+\alpha}$
by $\eps^2\sqrt{\log\eps^{-1}}.$
However we do not intend to  provide proofs  it  in  this paper. 
 \end{rem}


\subsection{Proof of Theorem \ref{thm:trois} in the two-dimensional case}

Arguing as in the proof of  \eqref{eq:strichartz1}, one may write
for all $p\geq2,$
\begin{eqnarray}\label{eq:N=2:1}
&&\|(b,z)\|_{L^\infty_t(H^s)}
\leq\frac32\|(b_0,z_0)\|_{H^s}\nonumber\\
&&\hspace{3cm}+C\Bigl(t^{\frac34}\|(b,v)_{\ell}\|_{L_t^4(C^{0,1})}
+t^{\frac1{p'}}\|(\eps\nabla b,v)_{h}\|_{L_t^p(C^{0,1})}\Bigr)
\|(b,z)\|_{L^\infty_t(H^s)}.\qquad
\end{eqnarray}
Therefore it remains to bound  $(b,v)_{\ell}$ in $L^4([0,T];C^{0,1})$
and $(\eps\nabla b,v)_{h}$ in $L^p([0,T];C^{0,1}).$
For the low-frequency part of the solution, 
 Proposition \ref{p:dispersive} ensures that
\begin{equation}\label{eq:N=2:2}
\|(b,v)_\ell\|_{L_t^4(C^{0,1})}
\leq C\eps^{\frac14}
\bigl(\|(b_0,v_0)_\ell\|_{B^{\frac74}_{2,1}}
+\|(f,g)_\ell\|_{L_t^1(B^{\frac74}_{2,1})}\bigr).
\end{equation}
The source terms $f_\ell$ and $g_\ell$ may be easily bounded  thanks to
Proposition \ref{p:besov}: we get 
\begin{equation}\label{eq:N=2:3}
\|(f,g)_\ell\|_{B^{\frac74}_{2,1}}\leq C\|(b,z)\|_{L^\infty}
\|(b,z)\|_{B^{\frac74}_{2,1}}.
\end{equation}
Let us now focus on  the high-frequency part of the solution.
Applying Proposition \ref{p:dispersive}  yields
\begin{equation}\label{eq:N=2:4}
\|(\eps Db,v)_h\|_{L_t^p(C^{0,1})}
\leq C\bigl(\|(\eps Db_0,v_0)_h\|_{B^{2-\frac2p}_{2,1}}
+\|(\eps Df,g)_h\|_{L_t^1(B^{2-\frac2p}_{2,1})}\bigr).
\end{equation}
Taking advantage of Proposition \ref{p:besov} and inequality
\eqref{eq:strichartzHF}, we get $$
\begin{array}{lll}
\|\eps\nabla f_h\|_{B^{2-\frac2p}_{2,1}}
&\leq& C\eps\bigl(\|b\|_{L^\infty}\|v\|_{B^{4-\frac2p}_{2,1}}
+\|v\|_{L^\infty}\|b\|_{B^{4-\frac2p}_{2,1}}\bigr),\\[2ex]
\|(g_1)_h\|_{B^{2-\frac2p}_{2,1}}&\leq& C\eps\||v|^2\|_{B^{4-\frac2p}_{2,1}}
\leq C\eps\|v\|_{L^\infty}\|v\|_{B^{4-\frac2p}_{2,1}},\\[2ex]
\|(g_2)_h\|_{B^{2-\frac2p}_{2,1}}
&\leq&C\eps\bigl(\|b\|_{L^\infty}\|\Im z\|_{B^{4-\frac2p}_{2,1}}
+\|\Im z\|_{L^\infty}\|b\|_{B^{4-\frac2p}_{2,1}}\bigr).
\end{array}
$$
 Following the computations leading to 
inequality \eqref{eq:strichartz2c}, it is not difficult to show that
$$
\|(\eps Db_0,v_0)_h\|_{B^{2-\frac2p}_{2,1}}
\leq C\eps\|(b_0,z_0)\|_{B^{3-\frac2p}_{2,1}}.
$$
If we assume that  $s>4-\frac2p$ then   inequalities 
\eqref{eq:N=3:2} to \eqref{eq:N=3:4} imply that
 $$\displaylines{
\eps^{-\frac14}\|(b,v)_\ell\|_{L_t^4(C^{0,1})}
+\eps^{-1}\|(\eps Db,v)_h\|_{L_t^p(C^{0,1})}
\leq C\Bigl(\|(b_0,v_0)\|_{H^s}\hfill\cr\hfill+\Bigl(t^{\frac34}
\|(b,v)_{\ell}\|_{L_t^4(C^{0,1})}+t^{\frac1{p'}}\|(\eps\nabla
b,v)_{h}\|_{L_t^p(C^{0,1})}\Bigr)\|(b,z)\|_{L^\infty_t(H^s)}\Bigr).}
$$
We introduce as before, for a  constant $c$ which is assumed to be sufficiently  small, the quantity 
$$
X(t):=\|(b,z)\|_{L^\infty_t(H^s)}
+c\eps^{-\frac14}\|(b,v)_\ell\|_{L_t^4(C^{0,1})}
+c\eps^{-1}\|(\eps Db,v)_h\|_{L_t^{p'}(C^{0,1})},
$$
we get
$$
X(t)\leq3X_0+C\bigl(\eps^{\frac14}t^{\frac34}
+\eps t^{\frac1{p'}}\bigr)X^2.
$$
It is now easy  to complete the proof of the theorem.
\qed


\subsection{Proof of Theorem \ref{thm:quatre}}

With Theorem \ref{thm:trois} at our disposal, we  compare the solution 
$(a_\eps,u_\eps)$ to the hydrodynamical form \eqref{eq:dynaslow}
of the Gross-Pitaevskii equation, to the solution 
$(\mathfrak a_\eps, \mathfrak u_\eps)$ of
the linear system 
$L_\eps(\mathfrak a_\eps,\mathfrak u_\eps)=0$
with the same initial datum. 
We notice that 
$(\tilde b,\tilde v)(x,t):=
(a_\eps-\mathfrak a_\eps,u_\eps-\mathfrak u_\eps)(x,\frac t\eps)$
 satisfies 
$$\left\{
\begin{array}{l}
\displaystyle
\partial_t\tilde b + \frac{\sqrt{2}}{\eps}{\rm div}\,\tilde v = f,
\\
\displaystyle
\partial_t\tilde v + \frac{\sqrt{2}}{\eps} \nabla\tilde b-\sqrt2\eps\nabla\Delta\tilde b = g
\end{array}
\right.
$$
with null initial datum,
 $f:=-\div(bv)$ and $g:=-\nabla(|v|^2)-\sqrt2\eps\nabla\div(b\,\Im z).$
\smallbreak
\noindent
By standard energy method, it follows that 
$$
\begin{array}{lll}
\|(\tilde b,\eps\nabla \tilde b,\tilde v)_\ell(t)\|_{H^{s-1}}&\leq& 
\Int_0^t\|(f,\eps\nabla f,g)_\ell(\tau)\|_{H^{s-1}}\,d\tau,\\[1.5ex]
\|(\tilde b,\eps\nabla\tilde b,\tilde v)_h(t)\|_{H^{s-2}}&\leq& 
\Int_0^t\|(f,\eps\nabla f,g)_h(\tau)\|_{H^{s-2}}\,d\tau.\end{array}
$$
Parseval equality entails that 
\begin{equation}\label{eq:cor0}
\|(b,\eps\nabla b,v)_\ell\|_{H^{s-1}}\approx \|(b,v)_\ell\|_{H^{s-1}}
\quad\hbox{and}\quad
\|(b,\eps\nabla b,v)_h\|_{H^{s-2}}\approx \eps \|b_h\|_{H^{s-1}}+\|v_h\|_{H^{s-2}},
\end{equation}
and a similar property holds for $(f,g)$.
 We remark that,  thanks  to the low frequency cut-off, we have
$$
\|\eps\nabla\div(b\,\Im z)_\ell\|_{H^{s-1}}\leq C\|\div(b\,\Im z)\|_{H^{s-1}}.
$$
Therefore, using Lemma \ref{lem:A.1}, we get
\begin{equation}\label{eq:cor1}
\|(f,g)_\ell\|_{H^{s-1}}\leq C\|(b,z)\|_{L^\infty}\|(b,z)\|_{H^s}.
\end{equation}
In order to bound $\|(\eps\nabla f,g)_h\|_{H^{s-2}},$ we use the fact that
$$
\|(\nabla|v|^2)_h\|_{H^{s-2}}\leq C\eps\||v|^2\|_{H^s}
$$
so that we end up with the inequality
\begin{equation}\label{eq:cor2}
\|(\eps\nabla f,g)_h\|_{H^{s-2}}\leq C\eps\|(b,z)\|_{L^\infty}\|(b,z)\|_{H^s}.
\end{equation}
Combining inequalities \eqref{eq:cor1} and \eqref{eq:cor2} and making use of 
\eqref{eq:cor0}, we obtain that 
\begin{equation}\label{eq:cor3}
\|\tilde b(t)\|_{H^{s-1}}+\|\tilde v_\ell(t)\|_{H^{s-1}} 
+\eps^{-1}\|\tilde v_h(t)\|_{H^{s-2}} \leq C\int_0^t
\|(b,z)(\tau)\|_{L^\infty}\|(b,z)\|_{H^s}\,d\tau.
\end{equation}
If we assume that $N\geq4$ then, according to inequality \eqref{eq:boot}, 
we have for some constant $C$ depending only 
on $s$ and on $N,$
$$
\|(b,z)\|_{L^\infty_t(H^s)}
+\eps^{-\frac12}\|(b,z)\|_{L^2_t(C^{0,1})}
\leq C\|(b_0,z_0)\|_{H^s}\quad\hbox{for all}\quad
t\in[0,T_\eps].
$$
Inserting the above inequality in \eqref{eq:cor3}
directly implies Theorem \ref{thm:quatre}
in the case $N\geq4.$

The conclusion in the case $N=2,3$ follows from similar arguments. The
details are left to the reader.
\qed

\goodbreak

\appendix

\section{Tame estimates}
We recall  several Gagliardo-Nirenberg type inequalities,
the proof of which may be found in \cite{Fri}, provide the proof to Lemma \ref{l:rouge}, and finally present a commutation  result. 
\begin{lem}\label{Gagliardo}
Let $k\in\N$ and $j\in\{0,\cdots,k\}$. There exists a constant $C_{j,k}$
depending only on $j$ and $k$ and such that 
the following inequality holds true:
$$
\|D^jv\|_{L^\frac{2k}j}
\leq C_{j,k}\|v\|_{L^\infty}^{1-\frac jk}
\|D^k v\|_{L^2}^{\frac jk}.
$$
\end{lem}
The Gagliardo-Nirenberg inequalities stated above
will enable us to prove the following
tame estimates for the product of two functions~:
\begin{lem}\label{lem:A.1}
Let $k\in\N$ and  $j\in\{0,\cdots,k\}$.
  There exists a constant  $C_{k,N}$ depending only on $(k,N)$, such that 
\begin{eqnarray}\label{tame1}
&\|D^ju\,D^{k-j}v\|_{L^2}
\leq C_{j,k}\Bigl(\|u\|_{L^\infty}\|D^{k}v\|_{L^2}
+\|v\|_{L^\infty}\|D^ku\|_{L^2}\Bigr),\\
\label{tame2}
&\|uv\|_{H^k}\leq C_k\Bigl(\|u\|_{L^\infty}\|u\|_{H^k}+
\|v\|_{L^\infty}\|u\|_{H^{k}}\Bigr).
\end{eqnarray}
\end{lem}
\begin{proof} 
Note that Leibniz formula combined with inequality \eqref{tame1}
yields \eqref{tame2}.
So let us prove the first inequality.
According to H\"older inequality, we have
$$
\|D^ju\,D^{k-j}v\|_{L^2}\leq \|D^ju\|_{L^{\frac{2k}j}}
\|D^{k-j}v\|_{L^{\frac{2k}{k-j}}}.
$$
This yields \eqref{tame1} if $j=0$ or $k$. Otherwise,  using
Lemma \ref{Gagliardo}, one can write that
$$\begin{array}{lll}
\|D^ju\,D^{k-j}v\|_{L^2}&\leq&C_{k,N}
\Bigl(\|u\|_{L^\infty}^{1-\frac jk}\|D^ku\|_{L^2}^{\frac jk}\Bigr)
\Bigl(\|v\|_{L^\infty}^{{\frac jk}}\|D^{k}v\|_{L^2}^{1-\frac jk}
\Bigr),\cr
&\leq&C_{k,N}
\Bigl(\|u\|_{L^\infty}\|D^{k}v\|_{L^{2}}\Bigr)^{1-\frac jk}
\Bigl(\|v\|_{L^\infty} \|D^ku\|_{L^2}\Bigr)^{\frac jk},
\end{array}
$$
and Young inequality  leads to $(\ref{tame1})$. 
\end{proof}
The tame estimates for the product of two functions extend
in every $H^s$ with $s\geq0$ and 
in the Besov space framework as follows (see the proof in e.g. \cite{BCD},
Chap. 2).
\begin{prop}\label{p:besov}
For any $r\in[1,\infty]$ and $s>0$ there exists a constant 
$C$ such that 
$$
\|uv\|_{B^s_{2,r}}\leq C\bigl(\|u\|_{L^\infty}\|v\|_{B^s_{2,r}}
+\|v\|_{L^\infty}\|u\|_{B^s_{2,r}}\bigr).
$$
\end{prop}
We also  recall the following continuity results
in Besov spaces for the left-composition
(see again e.g. \cite{BCD}, Chap. 2).
\begin{prop}\label{p:compo}
Let $F$ be a smooth function defined on some  open interval $I$
containing $0,$ and such that $F(0)=0.$
For any $r\in[1,\infty],$ $s>0$ and  compact subset $J$ of $I,$  there exists a
constant  $C$  such that for any function $u\in B^s_{2,r}$ valued in $J$ we
have  
$$\|F(u)\|_{B^s_{2,r}}\leq C\|u\|_{B^s_{2,r}}.
$$
\end{prop}

\noindent
{\bf Proof of  Lemma \ref{l:rouge}}.
We assume that condition \eqref{eq:strichartz0} holds. 
Using the fact that 
$v=v_\ell+v_h$ and Parseval formula, 
we easily get 
$$
\|v\|_{B^\sigma_{2,r}}\leq \|v_\ell\|_{B^\sigma_{2,r}}
+\|v_h\|_{B^\sigma_{2,r}}\leq 2\|v\|_{B^\sigma_{2,r}}.
$$
Next, because $b=b_\ell+b_h$ and
$\eps|\xi|\geq\eps_0$ for $\xi\in\Supp b_h,$
one may write
$$
\|b\|_{B^\sigma_{2,r}}\leq \|b_\ell\|_{B^\sigma_{2,r}}+\eps_0^{-1}
\|\eps\nabla b_h\|_{B^\sigma_{2,r}}\quad\hbox{and}\quad
\|b_\ell\|_{B^\sigma_{2,r}}\leq \|b\|_{B^\sigma_{2,r}}.
$$
According to \eqref{eq:pota}, we have $\Im z
=-\nabla\log(1+\eps b/\sqrt2).$
Therefore, condition \eqref{eq:strichartz0}
and Proposition \ref{p:compo} imply that
$$
\|\Im z\|_{B^\sigma_{2,r}}\leq \|\log(1+\eps b/\sqrt2)\|_{B^{\sigma+1}_{2,r}}
\leq C\eps\|b\|_{B^{\sigma+1}_{2,r}}.
$$
Now, using the definition of the $B^{\sigma+1}_{2,r}$ norm
and of $b_\ell, b_h,$ one may write
for a suitable $\eps_1>\eps_0,$
$$
\eps\|b\|_{B^{\sigma+1}_{2,r}}
\leq \sum_{\eps 2^q\leq\eps_1}(\eps 2^q)\,2^{q\sigma}\|\Delta_qb_\ell\|_{L^2}
+\sum_{\eps 2^q\geq\eps_0} 2^{q\sigma}\,(\eps2^q\|\Delta_qb_h\|_{L^2})
\leq C\bigl(\|b_\ell\|_{B^\sigma_{2,r}}+\|\eps\nabla
b_h\|_{B^\sigma_{2,r}}\bigr). $$
Therefore 
$$
\|(b,\Im z)\|_{B^\sigma_{2,r}}\leq C\bigl(\|b_\ell\|_{B^\sigma_{2,r}}+
\|\eps\nabla b_h\|_{B^\sigma_{2,r}}\bigr).
$$
In order to complete the proof of  \eqref{eq:equiv1}, 
we still have to show that
\begin{equation}\label{eq:ap1}
\|\eps\nabla b_h\|_{B^\sigma_{2,r}}\leq C \|(b,\Im z)\|_{B^\sigma_{2,r}}.
\end{equation}
In fact, as $L:z\mapsto\log(1+z)$ is a smooth diffeomorphism
from $(a,b)$ to $L((a,b))$ for any $0<a<b,$ and  vanishes at $0,$
Proposition \ref{p:compo} enables us to write that 
$$\begin{array}{lll}
\|\eps\nabla b_h\|_{B^\sigma_{2,r}}&\leq&\|\eps b\|_{B^{\sigma+1}_{2,r}},\\[1ex]
&\leq& C \|\log(1+\frac{\eps
b}{\sqrt2})\|_{B^{\sigma+1}_{2,r}},\\[1ex] &\leq&C\bigl(\|\log(1+\frac{\eps
b}{\sqrt2})\|_{B^{\sigma}_{2,r}} +\|\nabla\log(1+\frac{\eps
b}{\sqrt2})\|_{B^{\sigma}_{2,r}}\bigr),\\[1ex]
&\leq&C\bigl(\eps\|b\|_{B^{\sigma}_{2,r}}  +\|\Im
z\|_{B^{\sigma}_{2,r}}\bigr).\end{array} $$
This completes the proof of \eqref{eq:ap1} thus of \eqref{eq:equiv1}.
\medbreak
Let us now turn to the proof of inequality \eqref{eq:equiv2}.
Because $$
z=v-i\frac\eps{\sqrt2}\frac{\nabla b}{1+\frac\eps{\sqrt2}b}\quad\hbox{and}\quad
\nabla z=\nabla v-i\frac\eps{\sqrt2}\frac{\nabla^2 b}{1+\frac\eps{\sqrt2}b}
-i\frac{\eps^2}2\frac{|\nabla  b|^2}{\bigl(1+\frac\eps{\sqrt2}b\bigr)^2},
$$
condition \eqref{eq:strichartz0} guarantees that 
$$
\|(b,z)\|_{C^{0,1}}\leq \|v\|_{C^{0,1}}+C\eps\|\nabla b\|_{C^{0,1}}.
$$
Let us notice that, whenever $\tilde \chi\in{\mathcal C}_c^\infty$
has value $1$ on $\Supp\chi,$ 
one may write $b_\ell=\tilde\chi(\eps^{-1}D) b_\ell.$
Therefore,  there exists a $L^1(\R^N;\R^N)$ function $k$ 
so that
$$
(\eps\nabla b)_\ell=\eps\nabla\tilde\chi(\eps D) b_\ell
=\eps^{-N}k(\eps^{-1}\cdot)\star b_\ell.
$$
This ensures that
  $$
\eps\|\nabla b\|_{C^{0,1}}\leq \|(\eps\nabla b)_\ell\|_{C^{0,1}}
+\|(\eps\nabla b)_h\|_{C^{0,1}}\leq C\|b_\ell\|_{C^{0,1}}
+\|(\eps\nabla b)_h\|_{C^{0,1}}.
$$
so that
$$
\|(b,z)\|_{C^{0,1}}\leq C\bigl(\|(b,v)_\ell\|_{C^{0,1}}
+\|(\eps Db,v)_h\|_{C^{0,1}}\bigr).
$$
The reverse inequality follows from similar arguments. 
The details are left to the reader.\qed
\medbreak
The following commutation lemma  is central in  the proof of Proposition \ref{prop:1besov}.
\begin{lem}\label{l:alamormoilnoeud}
Let $s>0,$  $r\in[1,\infty]$ and $\psi$ be a smooth function compactly 
supported in an annulus $\{\xi\in\R^N\,/\, R_1\leq|\xi|\leq R_2\}.$
 There exists a constant $C$ depending only on $\psi$ and $s$ such that 
for all $q\in\N$ the 
 following estimate holds true: 
\begin{equation}\label{eq:com}  \|[a,\psi(2^{-q}D)]f\|_{L^2}\leq
Cc_q2^{-qs}\bigl(  \|Da\|_{L^\infty}\|f\|_{B^{s-1}_{2,r}}+
 \|Da\|_{B^{s-1}_{2,r}}\|f\|_{L^\infty}\bigr)
 \end{equation}
 for some sequence $(c_q)_{q\in\N}$ with $\|c_q\|_{\ell^r}=1.$
 
 A similar estimate is true with $q=0$ if
 $\psi$ is only supported in a ball.  
 \end{lem}
\begin{proof}
Decomposing $a$ into $a=S_0a+\tilde a,$ one may 
write
\begin{equation}\label{eq:com0}
[a,\psi(2^{-q}D)]f=[S_0a,\psi(2^{-q}D)]f+[\tilde a,\psi(2^{-q}D)]f.
\end{equation}
Remark that, owing to the support properties of $\psi,$ there exists
some integer $N_0$ such that
$$
[S_0a,\psi(2^{-q}D)]f=\sum_{|q'-q|\leq N_0}
[S_0a,\psi(2^{-q}D)]\Delta_{q'}f.
$$
Now, according
to Lemma 2.93 in \cite{BCD}, we have
$$
\|[S_0a,\psi(2^{-q}D)]\Delta_{q'}f\|_{L^2}\leq
C2^{-q}\|DS_0a\|_{L^\infty}\|\Delta_{q'}f\|_{L^2}, $$
whence, since $\|DS_0a\|_{L^\infty}\leq C\|Da\|_{L^\infty},$ 
$$
2^{qs}\|[S_0a,\psi(2^{-q}D)]f\|_{L^2}\leq C2^{-qs}\|Da\|_{L^\infty}
\sum_{|q'-q|\leq
N_0}2^{(q-q')(s-1)}\,\bigl(2^{q'(s-1)}\|\Delta_{q'}f\|_{L^2}\bigr) 
$$
so that we find that, for some sequence $(c_q)_{q\in\N}$ such that
$\|c_q\|_{q\in\N}=1$ and 
\begin{equation}\label{eq:com1}
\|[S_0a,\psi(2^{-q}D)]f\|_{L^2}\leq Cc_q2^{-qs}\|Da\|_{L^\infty}
\|f\|_{B^{s-1}_{2,r}}. 
\end{equation}
To deal  with the last term in \eqref{eq:com0}, 
one may take advantage of the paradifferential calculus based on a Littlewood-Paley decomposition, 
a tool introduced by J.-M. Bony in \cite{bony}. 
The paraproduct of two tempered distributions $u$ and $v$
is defined by 
$$
T_uv:=\sum_{q\geq1}S_{q-1}u\dq v
$$
and we have the following (formal) Bony's decomposition for the
product of two distributions:
$$
uv=T_uv+T'_vu\quad\hbox{with}\quad
T'_vu:=\sum_{q\geq-1}S_{q+2}v\Delta_{q}u.$$
This leads us to expand 
$[\tilde a,\psi(2^{-q}D)]f$ into
$$[\tilde a,\psi(2^{-q}D)]f=
[T_{\tilde a},\psi(2^{-q}D)]f+T'_{\psi(2^{-q}D) f}\tilde a-\psi(2^{-q}D)T'_f\tilde a.
$$
Taking advantage of the support properties of $\psi,$ one may write
for some suitable integer $N_0,$
$$
[T_{\tilde a},\dq]f=\sum_{\substack{q'\geq1\\|q'-q|\leq N_0}} [S_{q'-1}\tilde
a, \psi(2^{-q}D)] \Delta_{q'}f.
$$
Using again Lemma 2.93 in \cite{BCD}, one may write
$$
\|[S_{q'-1}\tilde
a, \psi(2^{-q}D)] \Delta_{q'}f\|_{L^2}
\leq C2^{-q}\|DS_{q'-1}\tilde a\|_{L^\infty}
\|\Delta_{q'}f\|_{L^2},
$$
so that we find 
\begin{equation}
\label{eq:com2}
\|[T_{\tilde a},\psi(2^{-q}D)]f\|_{L^2}\leq Cc_q2^{-qs}\|Da\|_{L^\infty}
\|f\|_{B^{s-1}_{2,r}}.
\end{equation}
Next, we have
$$
T'_{\psi(2^{-q}D)f}\tilde a=\sum_{q'\geq q-N_0} S_{q'+2}\psi(2^{-q}D)f\,
\Delta_{q'}\tilde a.
$$
Therefore
\begin{equation}\label{eq:com3}
\|T'_{\psi(2^{-q}D)f}\tilde a\|_{L^2}
\leq \sum_{q'\geq q-N_0} \|\psi(2^{-q}D)f\|_{L^2}\,
\|\Delta_{q'}\tilde a\|_{L^\infty}.
\end{equation}
Because ${\mathcal F}\tilde a$ is supported away from the origin, 
Bernstein inequality ensures that 
$\|\Delta_{q'}\tilde a\|_{L^\infty}\leq C2^{-q'}\|Da\|_{L^\infty}.$
Inserting this inequality in \eqref{eq:com3}, we thus get
$$
\|T'_{\psi(2^{-q}D)f}\tilde a\|_{L^2}
\leq C2^{-qs}\|Da\|_{L^\infty}\bigl(2^{q(s-1)}\|\psi(2^{-q}D)f\|_{L^2}\bigr)
\sum_{q'\geq q-N_0}2^{q-q'},
$$
whence
\begin{equation}\label{eq:com4}
\|T'_{\psi(2^{-q}D)f}\tilde a\|_{L^2}
\leq C \|Da\|_{L^\infty}\,2^{q(s-1)} \|\psi(2^{-q}D)f\|_{L^2}.
\end{equation}
Note that  because $\Delta_{q'}\psi(2^{-q}D)=0$ for $|q'-q|>N_0,$
there exists  some sequence $(c_q)_{q\in\N}$ with $\|c_q\|_{\ell^r}=1$
such that (see \cite[Section 2.7]{BCD})
\begin{equation}\label{eq:com4a}
2^{q(s-1)} \|\psi(2^{-q}D)f\|_{L^2}\leq Cc_q\|f\|_{B^{s-1}_{2,r}}.
\end{equation}
Finally, standard continuity results for the paraproduct\footnote{Here
we need that $s>0$, see \cite[Section 2.8]{BCD}} ensure
that  $$
\|T'_f\tilde a\|_{B^{s}_{2,r}}
\leq C \|f\|_{L^\infty}\|\tilde a\|_{B^{s}_{2,r}}
\leq C \|f\|_{L^\infty}\|Da\|_{B^{s-1}_{2,r}}
$$
so that, because $\Delta_{q'}\psi(2^{-q}D)=0$ for $|q'-q|>N_0,$
\begin{equation}\label{eq:com5}
\|\psi(2^{-q}D)T'_f\tilde a\|_{L^2}
\leq  C c_q2^{-qs}\|f\|_{L^\infty}\|Da\|_{B^{s-1}_{2,r}}.
\end{equation}
Putting together inequalities \eqref{eq:com1}--\eqref{eq:com5}
completes the proof.
\end{proof}


\section{Dispersive estimates}

This section is devoted to  the proof of  Proposition
\ref{p:dispersive}. 
We  first symmetrize system \eqref{eq:linearized}
by introducing the new  functions 
$$
c=(1-\eps^2\Delta)^{\frac12}b\quad\hbox{and}\quad d=(-\Delta)^{-\frac12}\div v, 
$$
and set  $F=(1-\eps^2\Delta)^{\frac12}f$ and $G=(-\Delta)^{-\frac12}\div g.$
If we restrict ourselves  to solutions $(b,v)$ such that $v$
is a gradient,  then system \eqref{eq:linearized} translates into
\begin{equation}\label{eq:cd}
\frac d{dt}\begin{pmatrix}c\\d\end{pmatrix}
=\begin{pmatrix} 0&-\eps^{-1}(-2\Delta)^{\frac12}\,(1-\eps^2\Delta)^{\frac12}\\
\eps^{-1}(-2\Delta)^{\frac12}\,(1-\eps^2\Delta)^{\frac12}&0\end{pmatrix}
\begin{pmatrix}c\\d\end{pmatrix}+
\begin{pmatrix}F\\G\end{pmatrix}.
\end{equation}
We recover our original  system \eqref{eq:linearized}  using  the inverse transformation  $b=(1-\eps^2\Delta)^{-\frac12}c$ and 
$v=-\nabla(-\Delta)^{-\frac12} d$. 
For $\eps>0$,  we  are hence led to consider the unitary group 
$(V_\eps(t))_{t\in\R}$
 on $L^2(\R^N)$  for the  infinitesimal generator
 $\mathcal I_\eps= i\eps^{-1}(-2\Delta)^{\frac12}(1-\eps^2\Delta)^{\frac12}$. As already  suggested in the introduction, when $\eps$ is large  $V_\eps$ behaves as the Schr\"odinger equation, whereas when $\eps$ is small, it behaves as  part of the wave system with speed $\sqrt{2}\slash \eps$. For this resaon, we introduce the slowed  operator 
 $$U_\eps(t)=V_\eps(\tfrac{\eps} {\sqrt{2}}t), $$
which should therefore behave as the wave operator of speed $1$.

\medbreak
The main ingredient for  the proof of the Strichartz estimates provided in  Proposition \ref{p:dispersive}  are the following uniform bounds.
\begin{lem}\label{l:dispersive}
Let $R_2>R_1>0$ and $a\in L^1(\R^N)$ such 
that $\Supp\hat a \subset\{\xi\in\R^N\,/\, R_1\leq|\xi|\leq R_2\}.$
 There exist two positive constants $\eps_1$  and  $C$
depending only on $N,R_1,R_2$ 
and such that  for all $t>0$ and $\eps\geq\eps_1,$
we have
\begin{equation}\label{eq:dispersive2}
\|V_\eps(t)a)\|_{L^\infty}\leq Ct^{-\frac{N}2}\|a\|_{L^1}.
\end{equation}
For all $\eps_0$ there exists a constant 
$C=C(\eps_0,N,R_1,R_2)$ such that for all $t>0$ and $\eps\in(0,\eps_0],$
we have
\begin{equation}\label{eq:dispersive1}
\|U_\eps(t)a)\|_{L^\infty}\leq Ct^{\frac{1-N}2}\|a\|_{L^1}.
\end{equation}
\end{lem}
{\it Proof}.
Fix some function $\phi\in{\mathcal C}^\infty_c(\R^N)$ supported 
in $\{R_1/2\leq|\xi|\leq 2R_2\}$ and with value $1$ on
$\{R_1\leq|\xi|\leq R_2\}.$
In view of  to the assumption  on $\Supp\hat a,$  we  may write
$$
U_\eps(t)a= (2\pi)^{-N}L_\eps(t)\star a\quad\hbox{and}\quad
V_\eps(t)a=(2\pi)^{-N} H_\eps(\sqrt 2\, t)\star a
$$
where we have set 
$$
L_\eps(t,x):=\!\int_{\R^N}e^{i(x\cdot\xi+t|\xi|\sqrt{1+(\eps|\xi|)^2})}\phi(\xi)\,d\xi
\!\quad\hbox{and}\!\quad
H_\eps(t,x):=\!\int_{\R^N}e^{i(x\cdot\xi+t|\xi|^2\sqrt{1+(\eps|\xi|)^{-2}})}\phi(\xi)\,d\xi.
$$
In order to prove the lemma, it suffices  therefore to establish  that
for all $\eps_0>0$ there  exists a constant $C$ such 
that for all  $t>0,$ we have
\begin{equation}\label{eq:dispersive3}
\|L_\eps(t)\|_{L^\infty}\leq Ct^{\frac{1-N}2}\quad\hbox{if}\quad
\eps\leq\eps_0,
\end{equation} 
and that there exists $\eps_1>0$ and a constant $C'$ such that
for all $t>0,$
\begin{equation}\label{eq:dispersive4}
\|H_\eps(t)\|_{L^\infty}\leq C't^{-\frac N2}\quad\hbox{if}\quad
\eps\geq\eps_1.
\end{equation}
As a matter if fact, inequalities \eqref{eq:dispersive3} and \eqref{eq:dispersive4} are derived from 
 the stationary and nonstationary phase theorems. 
The basic result that we shall invoke (see the proof in
e.g. \cite{BCD}, Chap. 8) reads.
 \begin{lem}\label{l:phase}
Let $K$ be a compact subset of $\R^N$ and $\psi$ be   a smooth function  supported
in $K.$ Let  $A$ be a real-valued smooth function defined
on some  neighborhood of $K.$ Set $$
I(t):=\int_{\R^N} e^{itA(\xi)}\psi(\xi)\,d\xi.
$$
For all couple $(k,k')$ of positive real numbers,  
there exists a constant $C$
depending only on $k,k'$ and on (a finite number of) derivatives of $A$ and $\psi$
such that for all $t>0,$
$$
|I(t)|\leq C\biggl(t^{-k}+\int_{\{\xi\in K\,/\,|\nabla A(\xi)|\leq1\}}
(1+t|\nabla A(\xi)|^2)^{-k'}\,d\xi\biggr).
$$
\end{lem}
\begin{proof}[Proof of Lemma \ref{l:dispersive} completed]
We first turn to inequality \eqref{eq:dispersive3}.
We notice that for any $x\in\R^N$ and $t>0,$ we have
$$
L_\eps(t,tx)=\int_{\R^N}e^{it(x\cdot\xi+|\xi|\sqrt{1+\eps^2|\xi|^2})}
\phi(\xi)\,d\xi.
$$
According to  Lemma \ref{l:phase}, we thus have for some constant $C$
depending  only on $N$ and~$\phi,$ 
\begin{equation}\label{eq:dispersive5}
|L_\eps(t,tx)|\leq C\biggl(t^{\frac{1-N}2}
+\int_{{\mathcal C}_x^\eps}\Bigl(1+t|\nabla A_x^\eps(\xi)|^2\Bigr)^{-N}\,d\xi\biggr) \end{equation}
where we have set 
$$
A_x^\eps(\xi):=x\cdot\xi+|\xi|\sqrt{1+\eps^2|\xi|^2}\quad\hbox{and}\quad
{\mathcal C}_x^\eps:=\bigl\{\xi\in\Supp\phi\,/\,  |\nabla A_x(\xi)|\leq1\bigr\}.
$$
We compute
$$
\nabla
A_x^\eps(\xi)=x+\biggl(\frac{1+2(\eps|\xi|)^2}{|\xi|\sqrt{1+(\eps|\xi|)^2}}\biggr)\xi.
$$
We may assume without  no loss of generality that $x\not=0,$ 
and decompose $\xi$ into
$$
\xi=\xi_x+\xi'_x\quad\hbox{where }\quad
\xi_x:=\Big ( \xi .  \frac{x}{|x|}\Big )\frac{x}{|x|},
$$
 so that we obtain,  for some positive constant $c$ depending only on  $\eps_0,$
$$
|\nabla A_x^\eps(\xi)|\geq c|\xi_x'|
\quad\hbox{for all}\quad \xi\in\Supp\phi\ \hbox{ and }\ \eps\in(0,\eps_0].
$$Plugging this inequality into  \eqref{eq:dispersive5}, one ends up with 
$$|L_\eps(t,tx)|\leq C\biggl(t^{\frac{1-N}2}
+\int_{-2R_2}^{2R_2}\int_{\R^{N-1}}\Bigl(1+t|\xi_x'|^2\Bigr)^{-N}
\,d\xi'_x\,d\xi_x\biggr).
$$ 
The change of variable $\eta'=\sqrt{t}\:\xi'_x$ finally 
yields \eqref{eq:dispersive3}. 
\smallbreak
For  the proof of  inequality \eqref{eq:dispersive4}, 
we use the fact that 
$$
H_\eps(t,tx)=\int_{\R^N} e^{it(x\cdot\xi+|\xi|^2\sqrt{1+(\eps|\xi|)^{-2}})}
\phi(\xi)\,d\xi.
$$
Using  Lemma \ref{l:phase}once more, we obtain that 
\begin{equation}\label{eq:dispersive6}
|H_\eps(t,tx)|\leq C\biggl(t^{-\frac N2}
+\int_{{\mathcal D}_x^\eps}\Bigl(1+t|\nabla
B_x^\eps(\xi)|^2\Bigr)^{-N}\,d\xi\biggr) \end{equation}
where we have set 
$$
B_x^\eps(\xi):=x\cdot\xi+|\xi|^2\sqrt{1+(\eps|\xi|)^{-2}}
\quad\hbox{and}\quad
{\mathcal D}_x^\eps:=\bigl\{\xi\in\Supp\phi\,/\,  |\nabla B_x(\xi)|\leq1\bigr\}.
$$
we write
$$
\nabla
B_x^\eps(\xi)=x+R_\eps(\xi)\xi\quad\hbox{where}\quad
R_\eps(\xi):=\frac{2+(\eps|\xi|)^{-2}}{\sqrt{1+(\eps|\xi|)^{-2}}}\cdotp
$$
Decomposing $\xi$ into $\xi=\xi_x+\xi'_x$ as before, and
using the fact that the integration is restricted 
to the set of   $R_1/2\leq|\xi|\leq2R_2,$ Lemma 
\ref{l:phase}  implies that  if $\eps\geq\eps_1>0$
then we have
$$
|H_\eps(t,tx)|\leq C\biggl(t^{-\frac N2}
+\int_{\frac{R_1}2<|\xi_x|<2 R_2}\int_{\R^{N-1}}
\Bigl(1+t\bigl(|\xi'_x|^2+(x+2\xi_x
R_\eps(\xi))^2\bigr)\Bigr)^{-N}\,d\xi'_x\,d\xi_x\biggr)  
$$
for some constant $C$ depending only on $\eps_1,N.$
If $\eps_1$ is  assumed to be sufficiently large,  then 
for all $\eps\geq\eps_1$ the map
$$
\Phi^\eps_x:\xi\longmapsto \sqrt t\bigl(x+\xi_x R_\eps(\xi)+\xi'_x\bigr)
$$
is a diffeomorphism from $\Omega:=\{\xi\in\R^N\,/\, R_1/2<|\xi|<2R_2\}$
 to $\Phi^\eps_x(\Omega)$  and that the jacobian of $\Phi_x^\eps$ is bounded
by below by $\alpha t^{N/2}$ for some $\alpha>0$
independent of $\eps.$
Making the change of variable $\eta=\Phi(\xi)$ in the above integral, 
we derive inequality \eqref{eq:dispersive4}.
\end{proof}

The above lemma will enable us  to prove 
Strichartz estimates  for the one-parameter unitary
group $(V_\eps(t))_{t\in\R}.$
Before we state  these estimates,  we recall the definition of 
 wave or Schr\"odinger admissible couples.
\begin{defi}
A couple of numbers  $(p,r)\in[2,\infty]^2$ is  said to be 
\begin{itemize} 
\item wave admissible if 
$$
\frac 1p+\frac{N-1}{2r}=\frac{N-1}4
\quad\hbox{and}\quad (p,r,N)\not=(2,\infty,3),
$$
\item Schr\"odinger admissible if 
$$
\frac 1p+\frac{N}{2r}=\frac{N}4
\quad\hbox{and}\quad (p,r,N)\not=(2,\infty,2).
$$
\end{itemize}
\end{defi}

If $p\geq 1$, we denote by $p'$ its H\"older conjugate exponent. As a consequence of Lemma \ref{l:dispersive} we have
\begin{cor}\label{c:strichartz}
Let $(p,r)$ and  $(p_1,r_1)$ be in $[2,\infty]^2,$  
 $a_0\in L^2(\R^N)$ and $f\in L^{p'_1}([0,T];L^{r'_1}(\R^N)).$
Assume in addition that $\hat a_0$ and  $\hat f(t,\cdot)$
are supported in $\{\xi\in\R^N\,/\,R_1\leq|\xi|\leq R_2\}.$

i) There exists $\eps_1=\eps_1(N,R_1,R_2)$  and a constant $C$ (independent of $T$)
such that if $(p,r)$ and $(p_1,r_1)$ are Schr\"odinger admissible then 
for all $\eps\geq\eps_1,$ 
$$\begin{array}{rcl}
\|V_\eps(t) a_0\|_{L_T^p(L^r)}&\leq& C\|u_0\|_{L^2},\\[1ex]
\Bigl\|\Int_0^t V_\eps(t-t') f(t')\,dt'\Bigr\|_{L_T^{p}(L^{r})}
&\leq& C\|f\|_{L_T^{p'_1}(L^{r'_1})}.
\end{array}
$$
ii) For all $\eps_0$ there exists a constant $C$ (independent of $T$) such 
that if $(p,r)$ and $(p_1,r_1)$ are wave admissible then 
for all $\eps\in(0,\eps_0],$ 
$$\begin{array}{rcl}
\|V_\eps(t) a_0\|_{L_T^p(L^r)}&\leq& C\eps^{\frac1p}\|u_0\|_{L^2},\\[1ex]
\Bigl\|\Int_0^t V_\eps(t-t') f(t')\,dt'\Bigr\|_{L_T^{p}(L^{r})}
&\leq& C\eps^{\frac1{p}+\frac1{p_1}} \|f\|_{L_T^{p'_1}(L^{r'_1})}.
\end{array}
$$

\end{cor}
\begin{proof} It follows from Lemma \ref{l:dispersive} and the fact that $V_\eps$   and $U_\eps$ are unitary  operators on $L^2(\R^N)$  that the assumptions of  the main result in \cite{KT} are met \footnote{For the choices $\sigma=\frac{N}{2} $ for $V_\eps$ and $\sigma=\frac{N-1}{2} $ for $U_\eps$, $\sigma$ being a parameter entering in the statement of \cite{KT}.}.  The conclusion of \cite{KT} yields i) for $V_\eps$ and $\eps\geq \eps_1.$ For statement ii), it suffices to rephrase the conclusion of \cite{KT} for $U_\eps$ in terms of $V_\eps$,  since
$$
V_\eps(t)a_0=U_\eps\Bigl(\textstyle{\frac{\sqrt2}{\eps}t}\Bigr)a_0\quad\hbox{and}\quad
\Int_0^t V_\eps(t-t') f(t')\,dt'=\frac\eps{\sqrt2}
\int_0^{\frac{\sqrt2}\eps\,t} U_\eps\Bigl(\textstyle{\frac\eps{\sqrt2}}t-\tau\Bigr)
f\Bigl(\textstyle{\frac\eps{\sqrt2}}\tau\Bigr)\,d\tau.\qedhere
$$
\end{proof}
\begin{lem}\label{l:dispersive2}
Let $(c,d)$ satisfy system $\eqref{eq:cd}$
with real-valued initial datum  $(c_0,d_0)$ and source terms $(F,G).$ 

i) For all $\eps_0>0$ and all wave admissible couples
$(p,r)$ and $(p_1,r_1)$ there exists a constant $C$ such 
that for all $q\in\Z$ and $\eps>0$ such that $2^q\eps\leq\eps_0$
we have
$$\displaylines{
2^{q\frac Nr}\|(\dq c,\dq d)\|_{L_T^p(L^r)}
\leq C\Bigl(\eps^{\frac1p}2^{q(\frac N2-\frac1p)}
\|(\dq c_0,\dq d_0)\|_{L^2}\hfill\cr\hfill 
+\eps^{\frac1p+\frac1{p_1}}
2^{q(\frac N{r'_1}-\frac1p-\frac1{p_1})}\|(\dq F,\dq G)\|_{L_T^{p'_1}(L^{r'_1})}\Bigr).}
$$

ii) There exists a constant $C$ such 
that for all $q\in\Z$ and $\eps>0$ such that $2^q\eps\geq\eps_1,$
and  all Schr\"odinger admissible couples
$(p,r)$ and $(p_1,r_1),$ 
we have
$$\!\!\!\!\!\!
2^{q\frac Nr}\|(\dq c,\dq d)\|_{L_T^p(L^r)}
\!\leq\! C\Bigl(2^{q(\frac N2\!-\!\frac2p)}
\|(\dq c_0,\dq d_0)\|_{L^2}+
2^{q(\frac N{r'_1}\!-\!\frac2p\!-\!\frac2{p_1})}\|(\dq F,\dq G)\|_{L_T^{p'_1}(L^{r'_1})}\Bigr).
$$
\end{lem}
\begin{proof}
Since  the data are real-valued, we have the identities
$$
\begin{array}{lll}
c(t)&\!\!=\!\!&\Re\bigl(V_\eps(t)c_0\bigr)-\Im\bigl(V_\eps(t)d_0\bigr)
+\Re\Int_0^tV_\eps(t-t')F(t')\,dt'-\Im\Int_0^t V_\eps(t-t')G(t')\,dt',\\[1ex]
d(t)&\!\!=\!\!&\Im\bigl(V_\eps(t)c_0\bigr)+\Re\bigl(V_\eps(t)d_0\bigr)
+\Im\Int_0^tV_\eps(t-t')F(t')\,dt'+\Re\Int_0^t V_\eps(t-t')G(t')\,dt.\end{array}
$$
Therefore,  we introduce the functions
$$
(\tilde c_q,\tilde d_q)(t,x):=(\dq c,\dq d)(2^{-2q}t,2^{-q}x)\!\!\quad\hbox{and}\!\!\quad
(\tilde F_q,\tilde G_q)(t,x):=2^{-2q}(\dq F,\dq G)(2^{-2q}t,2^{-q}x)
$$
so that $\tilde c_q,$ $\tilde d_q,$ $\tilde F_q$ and $\tilde G_q$
are  spectrally supported in $\{3/4\leq|\xi|\leq 8/3\},$
and  we have
$$\displaylines{
\tilde c_q(t)=\Re\bigl(V_{2^q\eps}(t)\tilde c_q(0)\bigr)-\Im\bigl(V_{2^q\eps}(t)\tilde d_q(0)\bigr)
\hfill\cr\hfill+\Re\Int_0^tV_{2^q\eps}(t-t')\tilde F_q(t')\,dt'
-\Im\Int_0^t V_{2^q\eps}(t-t')\tilde G_q(t')\,dt',}
$$
$$\displaylines{
\tilde d_q(t)=\Im\bigl(V_{2^q\eps}(t)\tilde c_q(0)\bigr)
+\Re\bigl(V_{2^q\eps}(t)\tilde d_q(0)\bigr)
\hfill\cr\hfill
+\Im\Int_0^tV_{2^q\eps}(t-t')\tilde F_q(t')\,dt'+\Re\Int_0^tV_{2^q\eps}(t-t')\tilde G_q(t')\,dt.}
$$
Next we fix some $\eps_0>0$.  Applying the first part of
Corollary \ref{c:strichartz}, we derive that for all 
wave admissible couples $(p,r)$ and $(p_1,r_1),$ and $\eps\in(0,\eps_0],$ we have
$$\displaylines{
\|(\tilde c_q,\tilde d_q)\|_{L_T^p(L^r)}
\leq C\Bigl((\eps2^q)^{\frac1p}
\|(\tilde c_q(0),\tilde d_q(0))\|_{L^2}+(\eps2^q)^{\frac1p+\frac1{p_1}}
\|(\tilde F_q,\tilde G_q)\|_{L_T^{p'_1}(L^{r'_1})}\Bigr).}
$$
Going  back to the initial variables, we  obtain  the desired
estimate for $(\dq c,\dq d).$

 The proof  of  the  inequality in the high-frequency regime
goes along the same lines: for this case, we use  instead  of the first part  the second part of Corollary \ref{c:strichartz}. 
\end{proof}

\begin{proof}[Proof of Proposition  \ref{p:dispersive} completed]
With Lemma \ref{l:dispersive2}  at our
disposal,  we  complete the proof of Proposition \ref{p:dispersive}.
Indeed, fix some smooth cut-off function $\chi$ 
with compact support and value $1$ on $B(0,\frac43\eps_1)$
and denote $z_\ell:=\chi(\eps^{-1}D)z$ and $z_h:=z-z_\ell$
 for any tempered distribution $z.$
 Owing to the spectral properties of $z_\ell$ and $z_h,$ 
 there exists  some $\eps_0>\eps_1$ such that 
 \begin{equation}\label{eq:dispersive7}
 \dq z_\ell=0\quad\hbox{for}\quad 2^q\eps>\eps_0\quad\hbox{and}
 \quad \dq z_h=0\ \hbox{ for }\ 2^q\eps<\eps_1.
 \end{equation}
 Let $(b,v)$ satisfy system \eqref{eq:linearized}.
 By virtue of \eqref{eq:dispersive6} and of  Bernstein inequality, 
 one may write for all $r\in[1,\infty],$
 \begin{equation}\label{eq:dispersive8}
 \|(b_\ell,v_\ell)\|_{L^\infty}\leq \sum_{2^q\eps\leq\eps_0}
 \|(\dq b_\ell,\dq v_\ell)\|_{L^\infty}
\leq C \sum_{2^q\eps\leq\eps_0}
 2^{q\frac Nr}\|(\dq b_\ell,\dq v_\ell)\|_{L^r}.
 \end{equation}
 Notice that as $\nabla|D|^{-1}$ and $|D|^{-1}\div$ are homogeneous
 multipliers of degree $0,$ we have (see e.g. Lemma 2.2 in \cite{BCD})
  \begin{equation}\label{eq:dispersive9}
   \|\dq v_\ell\|_{L^r}\approx \|\dq d_\ell\|_{L^r}.
   \end{equation}
  Next, we have $b_\ell=(1-\eps^2\Delta)^{\frac12}c_\ell$ and it is not difficult to 
 show that $A_\eps(D):=(1-\eps^2\Delta)^{\frac12}$ 
 and its inverse $ A_\eps^{-1}$
  are $S^0$-multipliers uniformly for $\eps\leq\eps_0$:
  for every  $k\in\N,$ there exists a constant $C_k$ such 
  that for every     $\eps\leq\eps_0$ and $\xi\in\R^N,$ we have
  $$
  |D^k A_\eps^{\pm1}(\xi)|\leq C_k(1+|\xi|^2)^{-k/2}.
  $$
  Therefore, a classical result    
   (see e.g. Lemma 2.2 in \cite{BCD}) ensures that 
   there exists a constant $C=C(N)$ such that
   for all $q\in\Z,$  $r\in[1,+\infty]$ and tempered distribution $z$
   we have
   \begin{equation}\label{eq:dispersive10}
   \|\dq(1-\eps^2\Delta)^{\pm\frac12}z\|_{L^r}\leq C\|\dq z\|_{L^r}
   \quad\hbox{for all }\ \eps\in[0,\eps_0].
   \end{equation}
   Combining these inequalities with \eqref{eq:dispersive8}, 
   we deduce that
    \begin{equation}\label{eq:dispersive11}
 \|(b_\ell,v_\ell)\|_{L^\infty}\leq C\sum_{2^q\eps\leq\eps_0}
    2^{q\frac Nr}\|(\dq c_\ell,\dq d_\ell)\|_{L^r}.
 \end{equation}
  Let us  consider first the case  $N\geq4.$ In this case,  we apply the first part of
Lemma \ref{l:dispersive2}
  with the wave admissible couples $(p,r):=(2,2(N-1)/(N-3)$ and $(p_1,r_1):=(\infty,2)$
   to deduce that
  $$
   \|(b_\ell,v_\ell)\|_{L^2_T(L^\infty)}\leq C\eps^{\frac12}\sum_{2^q\eps\leq\eps_0}
  2^{q(\frac N2-\frac12)}\Bigl(\|(\dq c_0,\dq d_0)_\ell\|_{L^2}
  +\|(\dq F,\dq G)_\ell\|_{L^1_T(L^2)}\Bigr).
  $$
  In order to  bound the r. h.s in terms of  the functions $b_0,$ $v_0$ $f$ and $g$, we invoke
   \eqref{eq:dispersive9}
  and \eqref{eq:dispersive10}. We end up with 
   \begin{equation}\label{eq:dispersive12}
   \|(b_\ell,v_\ell)\|_{L^2_T(L^\infty)}\leq C\eps^{\frac12}\sum_{2^q\eps\leq\eps_0}
  2^{q(\frac N2-\frac12)}\Bigl(\|(\dq b_0,\dq v_0)_\ell\|_{L^2}
  +\|(\dq f,\dq g)_\ell\|_{L^1_T(L^2)}\Bigr).
  \end{equation}
  Using similar arguments,  we get 
   $$
  \|(\nabla b,\nabla v)_\ell\|_{L^2_T(L^\infty)}\leq C\eps^{\frac12}\sum_{2^q\eps\leq\eps_0}
  2^{q(\frac N2-\frac12)}\Bigl(\|(\dq\nabla  b_0,\dq\nabla v_0)_\ell\|_{L^2}
  +\|(\dq\nabla f,\dq \nabla g)_\ell\|_{L^1_T(L^2)}\Bigr).
  $$
 Combining this latter inequality with \eqref{eq:dispersive12} and  using Bernstein inequality and the definition of 
the norm in $B^{\frac N2+\frac12}_{2,1},$  we  conclude that 
if $N\geq4$ then 
$$
\|(b,v)_\ell\|_{L^2_T(C^{0,1})}\leq C\eps^{\frac12}
\Bigl(\|(b_0,v_0)_\ell\|_{B^{\frac N2+\frac12}_{2,1}}+
\|(f,g)_\ell\|_{L_T^1(B^{\frac N2+\frac12}_{2,1})}\Bigr).
$$
If $N=3,$ the proof is almost the same except
that the endpoint couple $(2,\infty)$ is not admissible. 
However, we may  take  any couple $(p,r)$ 
with $1/p+1/r=1/2,$ and (as before) $(p_1,r_1)=(\infty,2).$
Applying Lemma \ref{l:dispersive2}, we get 
after a few computations, 
$$
\|(b,v)_\ell\|_{L^p_T(C^{0,1})}\leq C\eps^{\frac1p}
\Bigl(\|(b_0,v_0)_\ell\|_{B^{\frac52-\frac1p}_{2,1}}+
\|(f,g)_\ell\|_{L_T^1(B^{\frac 52-\frac1p}_{2,1})}\Bigr).
$$
Finally, in the case $N=2,$ one can take $(p,r)=(4,\infty)$
and  $(p_1,r_1)=(\infty,2).$ We end up with 
$$
\|(b,v)_\ell\|_{L^4_T(C^{0,1})}\leq C\eps^{\frac14}
\Bigl(\|(b_0,v_0)_\ell\|_{B^{\frac74}_{2,1}}+
\|(f,g)_\ell\|_{L_T^1(B^{\frac74}_{2,1})}\Bigr).
$$
  The part of Proposition \ref{p:dispersive} pertaining to the high frequencies
  of the solution may be proved exactly along the same lines.
  It suffices to apply the second part of Lemma \ref{l:dispersive2}.  
  The details are left to the reader.

\end{proof}

\end{document}